\documentclass[final,leqno,showlabe]{siamltex}
\usepackage{amsmath,graphicx}
\usepackage{stmaryrd}
\usepackage{booktabs}
\usepackage{cite}
\usepackage{amssymb}
\usepackage{verbatim}
\usepackage{hyperref}
\usepackage[all,cmtip]{xy}
\usepackage[top=3.54cm,bottom=4.54cm,left=2.94cm,right=2.94cm ]{geometry}
\usepackage[title]{appendix}
\usepackage{epstopdf}
\usepackage{algorithmic}
\usepackage{float}
\usepackage[caption=false]{subfig}

\usepackage{color}
\usepackage{enumerate}

\newcommand{\abs}[1]{\lvert#1\rvert}

\newcommand{\bes}{\begin{equation*}}
\newcommand{\ees}{\end{equation*}}

\input psfig.sty

\renewcommand{\theequation}{\arabic{section}.\arabic{equation}}
\renewcommand{\thetable}{\arabic{table}}
\renewcommand{\thefigure}{\arabic{figure}}

\newcommand{\eps}{\varepsilon}

\newtheorem{thm}{Theorem}[section]
\newtheorem{rmk}{Remark}[section]

\newcommand{\be}{\begin{equation}}
\newcommand{\ee}{\end{equation}}
\newcommand{\ba}{\begin{array}}
\newcommand{\ea}{\end{array}}
\newcommand{\bea}{\begin{eqnarray}}
\newcommand{\eea}{\end{eqnarray}}
\newcommand{\beas}{\begin{eqnarray*}}
\newcommand{\eeas}{\end{eqnarray*}}

\title{Error estimates of a regularized finite difference method for the Logarithmic Schr\"{o}dinger equation with Dirac delta potential\thanks{This work was partially support by NSFC grant 12171041 and 11771036 (Y. Cai). }
}
\author{ Xuanxuan Zhou \thanks{ School of Mathematics and Statistics, Nanjing University of Science and Technology, Nanjing,
		210094, China ({\tt xuanxuanzhou0208@njust.edu.cn})}
   \and Tingchun Wang \thanks{School of Mathematics and Statistics, Nanjing University of Information Science
and Technology, Nanjing, 210044, China
 ({\tt @wangtingchun2010@gmail.com})}
\and Yong Wu \thanks{Laboratory of Mathematics and Complex Systems,
Ministry of Education, School of Mathematical Sciences, Beijing Normal University,
   Beijing 100875, China ({\tt 202231130038@bnu.edu.cn})}
 \and Yongyong Cai\thanks{Laboratory of Mathematics and Complex Systems,
Ministry of Education, School of Mathematical Sciences, Beijing Normal University,
   Beijing 100875, China ({\tt yongyong.cai@bnu.edu.cn})}
}
\date{}

\begin{document}

\maketitle

\begin{abstract}
In this paper, we introduce a conservative Crank-Nicolson-type finite difference schemes for the regularized logarithmic Schr\"{o}dinger equation (RLSE) with Dirac delta potential in 1D. The regularized logarithmic Schr\"{o}dinger equation with a small regularized parameter $0<\eps \ll 1$ is adopted to approximate the logarithmic Schr\"{o}dinger equation (LSE) with linear convergence rate $O(\eps)$. The numerical method can be used to avoid numerical blow-up and/or to suppress round-off error due to the logarithmic nonlinearity in LSE. Then, by using domain-decomposition technique, 
we can transform the original problem  into an interface problem. Different treatments on the interface conditions lead to different discrete schemes and it turns out that a simple discrete approximation of the Dirac potential coincides with one of the conservative finite difference schemes. The optimal $H^1$ error estimates and the conservative properties of the finite difference schemes are investigated. The Crank-Nicolson finite difference methods enjoy the second-order convergence  rate in time and space. Numerical examples are provided to support our analysis and show the accuracy and efficiency of the numerical method.
\end{abstract}


\section{Introduction}
\label{sec:intro}
In this paper, we consider the logarithmic Schr\"{o}dinger equation \cite{Bose:2011,IJ:1976,IJ:1979,Hansson:2009,Hefter:1985,Hernandez:1980,Krolik:2000,Martino:2003,Yasue:1978,Zlosh:2010,Pava:2017} with a Dirac delta potential in one dimension (1D)
\begin{equation}\label{problem}
\begin{cases}
i\partial_t u(x,t)=-\partial_{xx} u(x,t)+\lambda \delta(x)u(x,t)+\mu u(x,t)\ln(|u(x,t)|^2),\quad x\in \mathbb{R},\,t>0,\\
u(x,0)=u_0(x),\quad x\in\mathbb{R},
\end{cases}
\end{equation}
where $u:=u(x,t)$ is a complex-valued function, $\delta(x)$ is the Dirac delta function, $\lambda,\mu\in\mathbb{R}$. When $\lambda=0$, the model\eqref{problem} has been used to study many physical problems, such as quantum mechanics \cite{Yasue:1978}, quantum optics \cite{IJ:1976,IJ:1979,Hansson:2009,Krolik:2000}, nuclear physics \cite{Hefter:1985,Hernandez:1980}, transport and diffusion phenomena \cite{Martino:2003}, open quantum systems, effective quantum gravity \cite{Zlosh:2010}, theory of superfluidity, and Bose-Einstein condensation \cite{Bose:2011}. For $\lambda\neq 0$, the model \eqref{problem} can be viewed as a model of a singular interaction between nonlinear wave and an inhomogeneity\cite{Pava:2017}. The delta potential can be used to model an impurity, or defect, localized at the origin. A similar formal model with power nonlinearity has been introduced in \cite{Aklan:2015,Sulem:1999}.  
In addition, the Dirac delta function $\delta(x)$ could be replaced by the  delta-comb potential $V(x)=\sum_{n}\delta(x-dn)$ ($d>0$).

It is well-known that the point interaction operator $H=-\partial_{xx}+\lambda \delta(x)$ ($\lambda\neq0$) is well-defined with domain $\{f(x)\in H^1(\mathbb{R})\cap H^2(\mathbb{R}\backslash\{0\})| \partial_x f(0+)-\partial_x f(0-)=\lambda f(0)\}$, where $\partial_x f(0\pm)=\lim\limits_{\varepsilon\to0^+}\frac{f(\pm\varepsilon)-f(0)}{\pm\varepsilon}$. The LSE \eqref{problem} possesses the following important invariants, i.e. mass, momentum and energy 

$$
\begin{aligned}
M(t) &: =\|u(\cdot, t)\|_{L^2(\Omega)}^2=\int_{\Omega}|u(x, t)|^2 d x \equiv \int_{\Omega}\left|u_0(x)\right|^2 d x=M(0) \\
P(t) & :=\operatorname{Im} \int_{\Omega} \bar{u}(x, t) \partial_x u(x, t) d x \equiv \operatorname{Im} \int_{\Omega} \overline{u_0}(x) \partial_x u_0(x) d x=P(0), \quad t \geq 0 \\
E(t) & :=\int_{\Omega}\left[|\partial_x u(x, t)|^2 d x+\lambda|u(0,t)|^2+\mu F\left(|u(x, t)|^2\right)\right] d x \\
& \equiv \int_{\Omega}\left[\left|\partial_x u_0(x)\right|^2+\lambda|u_0(0)|^2+\mu F\left(\left|u_0(x)\right|^2\right)\right] d x=E(0),
\end{aligned}
$$
where $\operatorname{Im} f$ and $\bar{f}$ denote the imaginary part and complex conjugate of $f$, respectively, and
$$
F(\rho)=\int_0^\rho \ln (s) d s=\rho \ln \rho-\rho, \quad \rho \geq 0 .
$$
In the literature, the well-posedness and the orbitally stability of the standing waves with a peak-Gausson porfile for the LSE with Dirac delta potential can be found in \cite{Pava:2017}. 

For numerically solving the LSE with Dirac delta potential \eqref{problem}, there are two main difficulties need to be considered. The first one is how to deal with the Dirac delta potential which has many methods been given such as: finite difference schemes \cite{Peskin:1977,Good:2004,Tornberg:2003}(by treating the Dirac delta function either as a single point discontinuity, or a smooth function with small support), the finite element methods \cite{Fem:1991,Holmer:2007,Holmer:2007-2}(the Dirac delta function is incorporated directly into the discretization via integration), interface methods \cite{IIM:1994,Peskin:2002,Rutka:2006,Ejiim:2000,Ejiimns:2008,Multi-sym:2018}, function transformation \cite{Fdm:2020}, etc\cite{Wavelet:2017}.

Hereafter, due to the singularity of the logarithmic nonlinearity, it introduces tremendous difficulties in establishing mathematical theories, as well as in designing and analyzing numerical methods . To address this issue, instead of working directly with \eqref{problem}, we shall consider the following RLSE with a small regularized parameter $0<\varepsilon \ll 1$ as
\begin{equation}\label{RLSE}
\left\{
\begin{array}{l}i \partial_t u^{\varepsilon}(\mathbf{x}, t)=-\partial_{xx} u^{\varepsilon}(\mathbf{x}, t)+\lambda \delta(x)u^{\varepsilon}(x,t)+ \mu u^{\varepsilon}(\mathbf{x}, t) \ln \left(\varepsilon+\left|u^{\varepsilon}(\mathbf{x}, t)\right|\right)^2, \quad \mathbf{x} \in \Omega, \quad t>0, \\ u^{\varepsilon}(\mathbf{x}, 0)=u_0(\mathbf{x}), \quad \mathbf{x} \in \bar{\Omega} .
\end{array}
\right.
\end{equation}
the corresponding conserved quantities, global well-posedness and the convergence of its solution to the solution \eqref{problem} can be found in \cite{Baosu:2019}.  
The mass, momentum, and "regularized" energy \cite{Cazenave:2003} are formally conserved for the RLSE \eqref{RLSE} :
\begin{align}
M^{\varepsilon}(t) & :=\int_{\Omega}\left|u^{\varepsilon}(x, t)\right|^2 d x \equiv \int_{\Omega}\left|u_0(x)\right|^2 d x=M(0), \label{conservation:M}\\
P^{\varepsilon}(t) & :=\operatorname{Im} \int_{\Omega} \overline{u^{\varepsilon}}(x, t) \partial_x u^{\varepsilon}(x, t) d x \equiv \operatorname{Im} \int_{\Omega} \overline{u_0}(x) \partial_x u_0(x) d x=P(0), \quad t \geq 0, \label{momentum:1}\\
E^{\varepsilon}(t) & :=\int_{\Omega}\left[\left|\partial_x u^{\varepsilon}(x, t)\right|^2+\lambda|u^{\eps}(0,t)|^2+\mu F_{\varepsilon}\left(\left|u^{\varepsilon}(x, t)\right|^2\right)\right] d x \label{conservation:E}\\
& \equiv \int_{\Omega}\left[\left|\partial_x u_0(x)\right|^2+\lambda|u_0(0)|^2+\mu F_{\varepsilon}\left(\left|u_0(x)\right|^2\right)\right] d x=E^{\varepsilon}(0),\nonumber
\end{align}
\text{where}
\begin{align*}
F_{\varepsilon}(\rho) & =\int_0^\rho \ln (\varepsilon+\sqrt{s})^2 d s \\
& =\rho \ln (\varepsilon+\sqrt{\rho})^2-\rho+2 \varepsilon \sqrt{\rho}-\varepsilon^2 \ln (1+\sqrt{\rho} / \varepsilon)^2, \quad \rho \geq 0 .
\end{align*}
For brevity, we still write $u^{\varepsilon}(x,t)$ as $u(x,t)$ in the following sections.

For numerically solving the LSE\eqref{RLSE} with $\lambda=0$, a semi-implicit finite difference method was proposed and analyzed in \cite{Baosu:2019}, and the $H^1$-error estimates is $\frac{e^{CT|\ln(\eps)|^2}}{\eps}(\tau^2+h^2)$. Later, in \cite{Baosu:2019_2}, Bao and Su presented the time-splitting spectral method and the Crank-Nicolson finite difference method (CNFD) for \eqref{RLSE} without Dirac delta potential, and establish rigorous $L_2$ error estimates $C\ln(\eps^{-1})\tau^{1/2}$ for the Lie-Trotter splitting under a much weaker constraint on the time step $\tau, \tau \lesssim 1 /|\ln (\varepsilon)|^2$, instead of $\tau \lesssim \sqrt{\varepsilon} e^{-C T|\ln (\varepsilon)|^2}$ \cite{Baosu:2019} for some $C$ independent of $\varepsilon$.
	In addition, Bao and Su  \cite{Baosu:2022} proposed a new systematic local energy regularization (LER) approach to overcome the singularity of the nonlinearity in the logarithmic Schr\"{o}dinger equation, and then presented and analyzed time-splitting schemes.




To the best of our knowledge, very few published work \cite{Wavelet:2017}  has been devoted to numerically solve the RLSE \eqref{RLSE} by using the domain decomposition methods\cite{Ddm:2020,Ddm:1999,Ddm:2004}. Thus, the main objective of this work is to study efficient discretization of the RLSE \eqref{RLSE} with delta-function potential, based on the interface formulation (or the domain decomposition type).  We shall construct and analyze efficient finite difference methods for solving the RLSE \eqref{RLSE}. Through analysis and numerical experiments for the interface formulation, the second order central finite difference scheme for RLSE \eqref{RLSE} is equivalent to approximate the Dirac delta function by a finite grid delta-like function, e.g. $\delta_h(x)=\frac{1-|x|/h}{h}$ ($|x|<h$, $h$ is the mesh size) and $\delta_h(x)=0$ ($|x|\ge h$). 

This paper is organized as follows.  In section \ref{sec:fds}, after the interface problem re-formulation of RLSE \eqref{RLSE}, a conservative finite difference methods will be constructed, and their corresponding conservation laws and error analysis are presented in section \ref{sec:erroranalysis}. Numerical experiments are presented in section \ref{sec:experiments} to verify the theoretical results and study the orbital stability of solitary waves. Finally, in section \ref{sec:conclusions}, we draw some conclusions and outlooks.

Throughout the paper, $C$ will denote a constant independent of the discretization parameters (mesh size $h$, time step size $\tau$ and time step number $n$) that may change from line to line. We use the notation $A\lesssim B$ to mean there exists a constant $C>0$ independent of the discretization parameters such that $|A|\leq C |B|$.

\section{Finite difference schemes}
\label{sec:fds}
For simplicity of notations and computational purpose, the NLS \eqref{problem} with a Dirac delta potential is truncated onto a bounded interval $\Omega=(-a,a)$ ($a>0$) with homogeneous Dirichlet boundary conditions(or periodic/homogeneous Neumann boundary conditions). The RLSE \eqref{RLSE} with the homogeneous Dirichlet boundary conditions collapses to
\begin{eqnarray}\label{problem_1}
\begin{aligned}
&i\partial_t u(x,t)=-\partial_{xx} u(x,t)+\lambda \delta(x)u(x,t)+\mu u(x,t)\ln(\varepsilon+|u(x,t)|)^2,\quad x\in\Omega,\,t>0,\\
&u(x,0)=u_0(x),\quad x\in\overline{\Omega};\quad u(-a,t)=u(a,t)=0,\quad t>0,
\end{aligned}
\end{eqnarray}
and the  corresponding conservation laws \eqref{conservation:M} and \eqref{conservation:E} hold. Similar to the whole space case \eqref{problem}, the domain of $H=-\partial_{xx}+\lambda \delta(x)$ is now  $\{f(x)\in H_0^1(\Omega)\cap H^2(\Omega\backslash\{0\})| \partial_x f(0+)-\partial_x f(0-)=\lambda f(0)\}$. Thus, it is natural to reformulate  \eqref{problem_1} into an equation defined on two subdomains ($(-a,0)$ and $(0,a)$) with the inner interface matching condition at  $x=0$ as
\begin{equation}\label{eq:jump}
\partial_x u(0+,t)-\partial_x u(0-,t)=\lambda u(0,t),\quad t>0.
\end{equation}
Using the interface condition \eqref{eq:jump}, the NLS  \eqref{problem_1} with homogeneous Dirichlet boundary conditions can be rewritten for $u_{l}(x,t):=u(x,t)$ ($x\in[-a,0]$) and $u_r(x,t):=u(x,t)$ ($x\in[0,a]$) as 
\begin{equation}\label{problem_1d}
\begin{cases}
i\partial_t u_{l}(x,t)=-\partial_{xx}u_{l}(x,t)+\mu u_l(x,t) \ln(\varepsilon+|u_l(x,t)|)^2,\quad x\in\Omega_l=(-a,0),\, t>0,\\
 i\partial_t u_{r}(x,t)=-\partial_{xx}u_{r}(x,t)+\mu u_r(x,t) \ln(\varepsilon+|u_r(x,t)|)^2,\quad x\in\Omega_r=(0,a),\, t>0,\\
u_l(x,0)=u_{0}(x),\,x\in\Omega_l; \quad u_r(x,0)=u_{0}(x),\,x\in\Omega_r;\quad u_l(-a,t)=0,
 u_r(a,t)=0,\, t>0, 
 \end{cases}
 \end{equation}
 with the interface condition \eqref{eq:jump} at $x=0$ as
 \begin{equation}\label{eq:interface}
u_l(0,t)=u_r(0,t),\quad \partial_x u_{r}(0,t)-\partial_x u_{l}(0,t)=\lambda u(0,t),\quad t>0.
\end{equation}
Now, it suffices to consider \eqref{problem_1d}-\eqref{eq:interface} for numerical purpose.
For a positive integer $M>0$, we choose the spatial mesh size $h=a/M$ and time step size $\Delta t=\tau$, and the grid points are $x_{j}=-a+jh$ ($j=0,1,\ldots, 2M$) with the time steps  $t_{k}=k \tau$ ($k=0,1,2,\cdots$).
Let $U_{j}^{k},U_{l,j}^k,U_{r,j}^k$ be the numerical approximation of the exact solution $u\left(x_{j}, t_{k}\right)$, $u_l\left(x_{j}, t_{k}\right)$, $u_r\left(x_{j}, t_{k}\right)$, respectively. 
Introduce the mesh function space 
\begin{equation*}
Z_h^0:=\{ u_j |j=0,\dots,2M,u_0=u_{2M}=0\}\in \mathbb{C}^{2M+1},
\end{equation*}
and we shall denote $U^k=(U_0^k,U_1^k,\cdots,U_{2M}^k)^T\in Z_h^0$ (also $U^k=(U_{l,0}^k,U_{l,1}^k,\cdots,U_{l,M}^k,U_{r,M+1}^k,\cdots,U_{r,2M}^k)^T$) as the numerical solution vector.
For any mesh function $U^k\in Z_h^0$ ($k\ge0$), the finite difference operators $\delta_t^+U^k,U^{k+1/2},\delta_{x}^\pm U^k,\delta_x U^k,\delta_x^2U^k\in Z_h^0$ are defined point-wisely  as
\begin{align*}
&\delta_{t}^+ U^{k}_j=\frac{1}{\tau}\left(U^{k+1}_j-U^{k}_j\right), \quad U^{k+\frac{1}{2}}_j=\frac{1}{2}\left(U^{k+1}_j+U^{k}_j\right),\\
&\delta_{x}^- U^k_{j}=\frac{1}{h}\left(U^k_{j}-U^k_{j-1}\right), \quad \delta_{x}^+ U^k_{j}=\frac{1}{h}\left(U^k_{j+1}-U^k_{j}\right), \\
&\delta_{x} U^k_{j}=\frac{1}{2h}\left(U^k_{j+1}-U^k_{j-1}\right),\quad\delta_{x}^{2} U^k_{j}=\frac{1}{h^2}\left(U^k_{j+1}-2U^k_{j}+U^k_{j-1}\right).
\end{align*}

For numerically solving \eqref{problem_1d}-\eqref{eq:interface}, inspired by \cite{Caizhou:2023}, we propose the following Crank-Nicolson finite difference  scheme (CNFD):\\
\noindent {\bf{CNFD:}} Given $U_{l,j}^0=u_0(x_j)$ ($0\leq j\leq M$) and $U_{r,j}^0=u_0(x_j)$ ($M\leq j\leq 2M$), for $k\ge0$, $U^{k+1}\in Z_h^0$ is given by
\begin{equation}\label{eq:cnfd}
\begin{cases}
i\delta_t^+U_{l,j}^{k}=-\delta_{x}^2U_{l,j}^{k+1/2}+\mu\tilde{Q}(U_{l,j}^{k+1},U_{l,j}^{k}),\quad j=1,\dots,M-1,\\
i\delta_t^+U_{r,j}^{k}=-\delta_{x}^2U_{r,j}^{k+1/2}+\mu\tilde{Q}(U_{r,j}^{k+1},U_{r,j}^{k}),\quad j=M+1,\dots,2M-1,\\
i\delta_t^+U_{l,M}^{k}=-\frac{1}{h^2}(U_{l,M+1}^{k+1/2}-2U_{l,M}^{k+1/2}+U_{l,M-1}^{k+1/2})+\mu\tilde{Q}(U_{l,M}^{k+1},U_{l,M}^{k}),\\
i\delta_t^+U_{r,M}^{k}=-\frac{1}{h^2}(U_{r,M+1}^{k+1/2}-2U_{r,M}^{k+1/2}+U_{r,M-1}^{k+1/2})+\mu\tilde{Q}(U_{r,M}^{k+1},U_{r,M}^{k}),\\
U_{l,M}^{k+1/2}=U_{r,M}^{k+1/2},\quad \delta_{x}U_{r,M}^{k+1/2}-\delta_{x}U_{l,M}^{k+1/2}=\lambda U_{M}^{k+1/2},\quad U_{l,0}^{k+1}=0,\quad U_{r,2M}^{k+1}=0,\quad k\ge0.
\end{cases}
\end{equation}
where $\tilde{Q}(z_1,z_2)$ for $z_1,z_2\in\mathbb{C}$ is defined as
\begin{equation}\label{eq:tildeQ}
\tilde{Q}(z_1,z_2)=\frac{
1}{2}\frac{Q(|z_1|^2)-Q(|z_2|^2)}{|z_1|^2-|z_2|^2}(z_1+z_2)=\frac{1}{2}\int_0^1q(\theta|z_1|^2+(1-\theta)|z_2|^2)\,d\theta\, (z_1+z_2),
\end{equation}
and $Q(s)=\int_{0}^{s}q(s_1)ds_1=2\left(s-\varepsilon^2\right) \ln (\varepsilon+\sqrt{s})- s+2 \varepsilon  \sqrt{s},\, q(s)= \ln (\varepsilon+\sqrt{s})^2$ ($s\ge0$).  

CNFD \eqref{eq:cnfd} is a two-level implicit and conservative scheme, and the interface jump condition \eqref{eq:interface} is approximated at the second-order accuracy ($\delta_x u_{l}(x_M,t),\delta_x u_{r}(x_M,t)$ is a second order approximation of ($\partial_x u_l(x_M,t),\partial_x u_r(x_M,t)$). 

Using the condition $U_{l,M}^k=U_{r,M}^k$, for CNFD \eqref{eq:cnfd}, the solution vector $U^k\in Z_h^0$ ($k\ge0$) is given by
\begin{equation}\label{U:def}
U^k=(U_{l,0}^k,U_{l,1}^k,\dots,U_{l,M-1}^k,U_{l,M}^k,U_{r,M+1}^k,\dots,U_{r,2M-1}^k,U_{r,2M}^k)^T.
\end{equation}
Introducing the index sets
\begin{equation}\label{Index:def}
\mathcal{T}_M^0=\{0,1,2,\cdots,2M-1\},\quad \mathcal{T}_M^{0,1}=\mathcal{T}_M^0\backslash\{M\},\quad 
\mathcal{T}_M=\{1,2,\cdots,2M-1\},\quad \mathcal{T}_M^1=\mathcal{T}_M\backslash\{M\},
\end{equation}
 CNFD can be reformulated for $U^k\in Z_h^0$ ($k\ge0$, $U^0_j=u_0(x_j)$, $0\leq j\leq 2M$) in the equivalents forms:
\begin{align}
&i\delta_t^+U_{j}^{k}=-\delta_{x}^2U_{j}^{k+1/2}+\lambda b_j U_j^{k+1/2}+\mu\tilde{Q}(U_j^{k+1},U_j^{k}),\quad \,j\in\mathcal{T}_M,\label{Fdm_2-1}\\
&U_0^{k+1}=U_{2M}^{k+1}=0, \quad k\ge0,\label{Fdm_2-2}
\end{align}
where $b=(b_0,b_1,\cdots,b_{2M})^T\in Z_h^0$ is given by
\begin{equation}
b_j=
\begin{cases}
0,&j\neq M,\\
\frac{1}{h},&j=M.
\end{cases}
\end{equation}

\begin{rmk}
For deriving CNFD \eqref{Fdm_2-1}-\eqref{Fdm_2-2} from \eqref{eq:cnfd}, we have used the  two approximations at the point $x_M=0$,  such that
\begin{equation*}
2i\delta_t^+U_{M}^{k}=-\frac{1}{h^2}(U_{l,M+1}^{k+1/2}-2U_{M}^{k+1/2}+U_{M-1}^{k+1/2}+U_{M+1}^{k+1/2}-2U_{M}^{k+1/2}+U_{r,M-1}^{k+1/2})+2\mu\tilde{Q}(U_M^{k+1},U_M^{k}).
\end{equation*}
Recalling the interface condition which implies $U_{l,M+1}^{k+1/2}+U_{r,M-1}^{k+1/2}=-2\lambda hU_M^{k+1/2}+U_{M-1}^{k+1/2}+U_{M+1}^{k+1/2}$, \eqref{Fdm_2-1} follows.  It is easy to observe that if we approximate the delta potential in \eqref{problem_1d} with a regular function $\delta_h(x)$ such that $\delta_h(x)=\frac{1-|x|/h}{h}$ ($|x|<h$) and $\delta_h(x)=0$ ($|x|\ge h$), the conventional Crank-Nicolson type scheme would  lead to the same discretization \eqref{Fdm_2-1}-\eqref{Fdm_2-2}.
\end{rmk}

Before presenting the discrete conservation properties  and error analysis of the constructed scheme CNFD \eqref{Fdm_2-1}-\eqref{Fdm_2-2}, some useful definitions and lemmas are introduced. For any mesh functions $U, V\in Z_h^0$, we introduce the $l^2$ inner product and the norms ($l^2$, $l^\infty$, $l^q$ ($1<q<+\infty$), semi-$H^1$),
\begin{align*}
&(U, V)=h\sum_{j=0}^{2M-1} U_{j} \overline{V}_{j},\quad\|U\|=\sqrt{(U, U)}, 
\quad\|U\|_{\infty}=\max _{0 \leq j \leq 2M}\left|U_{j}\right|,\\
&\|U\|_{q}^q=h\sum_{j=0}^{2M-1} |U_{j}|^q,\quad |U|_{H^1}=\left(\delta_{x}^+ U, \delta_{x}^+ U\right), 
\end{align*}
where $\bar{c}$ is the complex conjugate of $c\in\mathbb{C}$.
The following summation-by-parts lemma is useful.
\begin{lemma}\label{delta_x}\cite{Caizhou:2023}
For any mesh functions $U,V\in Z_h^0$, there hold
\begin{align*}
&(\delta_x^2 U,V)=-(\delta_x^+ U,\delta_x^+ V),\quad (\delta_x^2 U,U)=-\|\delta_x^+ U\|^2,\\
&(\delta_x^2 U,V)-h\delta_x^2U_M \overline{V}_M=-(\delta_x^+ U,\delta_x^+ V)+\frac{1}{h}(2U_M\overline{V_M}-\overline{V_M}(U_{M+1}+U_{M-1})),\\
&(\delta_x^2 U,U)-h\delta_x^2U_M \overline{U}_M=-\|\delta_x^+ U\|^2+\frac{1}{h}(2U_M\overline{U_M}-\overline{U_M}(U_{M+1}+U_{M-1})).
\end{align*}
\end{lemma}
Similar to the NLS \eqref{problem_1d} on the continuous level, the proposed CNFD \eqref{Fdm_2-1}-\eqref{Fdm_2-2} have the following conservation properties.
\begin{lemma}\label{conservation} 
 CNFD \eqref{Fdm_2-1}-\eqref{Fdm_2-2} satisfies the conservation laws:
\begin{align}
&E^{k}:=\|\delta_x^+U^{k}\|^2+ \lambda|U_{M}^{k}|^2+ \mu h\sum\limits_{j\in\mathcal{T}_M^0}Q(|U_j^{k}|^2)=E^0,\label{energy}\\
&M^{k}:=\|U^k\|^2=M^0.\label{mass}
\end{align}
\end{lemma}
\begin{proof}
Firstly, multiplying \eqref{Fdm_2-1} by $h\overline{U_j^{k+1/2}}$ ($j\in\mathcal{T}_M^0$), summing all together and taking the imaginary parts, noticing the Lemma \ref{delta_x} , we can obtain
\begin{align*}
\frac{1}{2\tau}\left(M^{k+1}-M^k\right)&=\mu \text{Im}(U_M^{k+1/2}\overline{U_M^{k+1/2}})- \text{Im}(\delta_x^2U^{k+1/2},U^{k+1/2})=0,
\end{align*}
and the mass conservation holds. For the energy conservation, multiplying both sides of \eqref{Fdm_2-1} by $h(\overline{U_j^{k+1}-U_j^{k}})$ ($j\in\mathcal{T}_M^0$),   summing all together and taking the real parts, we have
\begin{align*}
\frac{h}{\tau}\text{Im}\sum\limits_{j\in\mathcal{T}_M^0}(U_j^{k+1}-U_j^{k})(\overline{U_j^{k+1}-U_j^{k}})&=\lambda \text{Re}\,(U_M^{k+1/2}(\overline{U_M^{k+1}-U_M^k}))- \text{Re}(\delta_x^2U^{k+1/2},U^{k+1}-U^k)\\
&+h\mu\text{Re}\,\sum\limits_{j\in\mathcal{T}_M^0}\tilde{Q}(U_j^{k+1},U_j^{k})(\overline{U_j^{k+1}-U_j^{k}}).
\end{align*}
Noticing  Lemma \ref{delta_x}, a direct calculation leads to 
\begin{align*}
\|\delta_x^+U^{k+1}\|^2+\lambda|U_{M}^{k+1}|^2+ \mu h\sum\limits_{j\in\mathcal{T}_M^0}Q(|U_j^{k+1}|^2)
=\|\delta_x^+U^{k}\|^2+\lambda|U_{M}^{k}|^2+ \mu h\sum\limits_{j\in\mathcal{T}_M^0}Q(|U_j^{k}|^2),
\end{align*}
i.e.  the energy conservation \eqref{energy} holds.
\end{proof}
\begin{rmk}\label{rmk:apriori}
Lemma \ref{conservation} implies the {\it a priori} $l^\infty$ bound of the numerical solution $U^k$  ($k\ge0$, see Appendix \ref{appendix:A} for details), for CNFD, i.e.
there exists $C>0$ independent of $h$, $\tau$ and $k$ such that $\|U^k\|_\infty\leq C$. 
\end{rmk}

\section{Error estimates}
\label{sec:erroranalysis}
Below, based on the energy method, we shall rigorously establish the error estimates of the proposed CNFD \eqref{Fdm_2-1}-\eqref{Fdm_2-2} in the fully-discrete form. Let $0<T<T^*$ with $T^*$ being the  maximal existence time of the solution $u(x,t)$ to the nonlinear Regularized Logarithmic Schr\"{o}dinger equation \eqref{RLSE} for a fixed $0\leq \varepsilon\ll 1$. Motivated by the analytical results in \cite{Baosu:2019,Caizhou:2023}, we assume $ u(x,t)\in L^{\infty}\left([0,T], H_0^1(\Omega)\cap H^2(\Omega\backslash\{0\})\right),u_x(0+)-u_x(0-)=\lambda u(0)$, and 

\begin{equation*}(\text{A})
\begin{split}
&\|u\|_{W^{1,\infty}([0,T],W^{4,\infty}(\Omega\backslash\{0\}))}
+\|u\|_{ W^{2,\infty}([0,T],L^{\infty}(\Omega\backslash\{0\}))} \\
&+
\|u\|_{ W^{3,\infty}([0,T],W^{2,\infty}(\Omega\backslash\{0\}))}
+\|u\|_{ W^{4,\infty}([0,T],L^{\infty}(\Omega\backslash\{0\})) }\leq M_1,
\end{split}
\end{equation*}

Define the error function $e^k\in Z_h^0$ ($k=0,1,2,\cdots$) as
\begin{align*}
&e_{j}^{k}=u(x_j,t_k)-U_{j}^{k}, \quad 0 \leq j \leq 2M.
\end{align*}
Then we have the following convergence results for the finite difference discretizations.
\begin{thm}\label{convergence_fdm1}
Under the assumption (A), there exist $h_0>0$ and $\tau_0>0$ sufficiently small, when $0<h\leq h_0,0<\tau\leq \tau_0$, CNFD \eqref{Fdm_2-1}-\eqref{Fdm_2-2} satisfies the following error estimates:
\begin{eqnarray}\label{convergence_2}
\begin{aligned}
&\left\|e^{k}\right\| \leq C_{\varepsilon}(\tau^{2}+\tau h)+Ch^2,\,
\left\|\delta_x^+e^{k}\right\| \leq C_{\varepsilon}(\tau^{2}+\tau h)+Ch^2,\\
&\left\|e^{k}\right\|_{\infty} \leq C_{\varepsilon}(\tau^{2}+\tau h)+Ch^2,  k\geq 0.
\end{aligned}
\end{eqnarray}
where $C_{\varepsilon}=M_1\max{(|\ln(\eps+M_1)|,\frac{1}{M_1^4},\varepsilon^{-2})}$.
\end{thm}

\subsection{Error estimates of CNFD}
Define the local truncation error $R^k\in Z^0_h$ ($k\ge0$) for CNFD \eqref{Fdm_2-1}-\eqref{Fdm_2-2} as
\begin{equation}\label{eq:lc-2}
R_j^k=i\delta_t^+u(x_j,t_k)+\left(\delta_{x}^2-\lambda b_j\right) \frac{u(x_j,t_k)+u(x_j,t_{k+1})}{2}-\mu\tilde{Q}(u(x_j,t_{k+1}),u(x_j,t_{k})),\quad \,j\in\mathcal{T}_M.
\end{equation}

\begin{lemma}\label{Truncation_error_2}
Under the assumption (A), we have the estimates for the local error $R^k\in Z_h^0$ ($0\leq k\leq T/\tau-1$) in \eqref{eq:lc-2}:
\begin{align*}
&|R_j^{k}|\leq C_{\varepsilon}\tau^{2}+Ch^2,\quad |\delta_t^+ R_j^{k}|\leq C_{\varepsilon}\tau^{2}+Ch^2,\quad j\in\mathcal{T}_M^0\backslash \{M-1\};\\
&|\delta_t^+ R_M^{k}|\leq C_{\varepsilon}\tau^{2}+Ch,\quad|R_M^{k}|\leq C_{\varepsilon}\tau^{2}+Ch,
\end{align*}
where $C_{\varepsilon}=M_1\max{(|\ln(\eps+M_1)|,\frac{1}{M_1^4},\varepsilon^{-2})}$. In addition, we have
\begin{equation}
\|R^k\|\lesssim C_{\varepsilon}\tau^{2}+Ch,\quad \|\delta_t^+ R^k\|\lesssim C_{\varepsilon}\tau^{2}+Ch.
\end{equation}
\end{lemma}
\begin{proof}The case with $j\neq M$ is standard and we focus on $j=M$. For $x_M=0$, applying the Taylor expansion and the interface condition, under the assumption (A) ($u(x,t)$ is piece-wisely  regular), we have
\begin{align}
R_M^k=
&\frac{\tau^2}{16}\int_{0}^{1}\int_{-\theta}^{\theta}\left(\partial_x^2\partial_t^2 u(0+,t_{k+\frac{1}{2}}+\frac{s\tau}{2})+\partial_x^2\partial_t^2 u(0-,t_{k+\frac{1}{2}}+\frac{s\tau}{2})\right)
ds\,d\theta\nonumber\\
&+\frac{i\tau^2}{8}\int_{0}^{1}\int_{0}^{\theta}\int_{-s}^{s}\partial_t^3 u(0,t_{k+\frac{1}{2}}+\frac{\sigma\tau}{2})d\sigma\,ds\,d\theta\label{eq:RMk2}\\
&+\frac{h}{2}\int_{0}^{1}\int_0^{\theta}\int_0^s\sum_{m=0,1}\left(\partial_x^3 u(\sigma h,t_{k}+m\tau)-\partial_x^3 u(-\sigma h,t_{k}+m\tau)\right)\, d\sigma\,ds\,d\theta\nonumber
\\
&+\mu\frac{\tau^2}{8}\int_{0}^{1}\int_{-\theta}^{\theta}\partial_t^2 u_l(0,t_{k+\frac{1}{2}}+\frac{s\tau}{2})dsd\theta\,\cdot q_0^k +\mu u(0,t_{k+\frac{1}{2}})r_0^k,\nonumber\\
&+i\partial_{t}u(0,t_{k+1/2})+\frac{1}{2}(\partial_{xx}u(0+,t_{k+1/2})+\partial_{xx}u(0-,t_{1+1/2}))-\mu\ln(\eps+|u(0,t_{k+\frac{1}{2}}|)^2u(0,t_{k+\frac{1}{2}}),\nonumber
\end{align}
where $q_0^k=\int_0^1q(\sigma|u(0,t_{k+1})|^2+(1-\sigma)|u(0,t_{k})|^2)\,d\sigma$ and $r_0^k=q_0^k-\ln(\eps+|u(0,t_{k+\frac{1}{2}}|)^2$. 
Recalling $q(s)=\ln(\eps+\sqrt{s})^2$ \eqref{eq:tildeQ}, we have 
\begin{equation*}
|q_0^k|\leq 2\int_0^1\left|\ln(\eps+ \sqrt{\sigma|u(0,t_{k+1})|^2+(1-\sigma)|u(0,t_{k})|^2})\right|\,d\sigma
\leq 2\max\{\ln(\eps^{-1}),|\ln(\eps+M_1)|\}.
\end{equation*}
Choosing the middle-point rule for the integral in $q_0^k$, we have
$q_0^k=\tilde{q}_0^k+O((|u(0,t_k)|^2-u(0,t_{k+1})|^2)^2)=\tilde{q}_0^k+\max{(M_1^{-4},\varepsilon^{-2})}\tau^2$ with $\tilde{q}_0^k=q((|u(0,t_k)|^2+|u(0,t_{k+1})|^2)/2)$.  Using $(|u(0,t_k)|^2+|u(0,t_{k+1})|^2)/2=|u(0,t_{k+1/2})|^2+\max{(M_1^{-2},\varepsilon^{-1})}\tau^2$ (central difference), we have
$r_0^k=q_0^k-\tilde{q}_0^k+\tilde{q}_0^k-\ln(\eps+|u(0,t_{k+\frac{1}{2}}|)^2=\max{(M_1^{-4},\varepsilon^{-2})}\tau^2$ (see Appendix \eqref{appendix_B} for details). Based on the assumption (A), the last line in \eqref{eq:RMk2} vanishes and we arrive at $|R_M^k|=C_{\varepsilon}\tau^{2}+Ch$. The case for $\delta_t^+R_M^k$ is similar (see \cite{Baocai:2012,Baocai:2013,Baosu:2019})(see Appendix \eqref{appendix_B} for details), and the detail is omitted here for brevity.
\end{proof}

Subtracting  \eqref{Fdm_2-1} from \eqref{eq:lc-2},  we have the error equation for CNFD as 
\begin{align}
&i\delta_t^+e_j^{k}=-\delta_x^2e_j^{k+1/2}+\lambda b_j e_j^{k+1/2}+\mu\hat{Q}_j^k+R_j^{k},\quad j\in\mathcal{T}_M,\,k\ge0,\label{err2_1}
\end{align}
 where $\hat{Q}^k\in Z_h^0$ and
\begin{equation}\label{eq:QI2}
\hat{Q}_j^k=\tilde{Q}\left(u(x_j,t_{k+1}),u(x_j,t_{k})\right)-\tilde{Q}(U_j^{k+1},U_j^{k})\quad j\in\mathcal{T}_M.
\end{equation}

For the nonlinear term $\hat{Q}_j^k$ ($k\ge0$, $j\in\mathcal{T}_M^1$), based on Remark \ref{rmk:apriori}, under the appropriate assumptions, we have the $l^\infty$ bounds of the numerical solution $U^k$ and the following estimates hold.
\begin{lemma}\label{Nonlinearity_1}
Under the assumption (A),  $\hat{Q}_j^k$ ($k\ge0$, $j\in\mathcal{T}_M^1$) \eqref{eq:QI2} satisfies
\begin{align}\label{eq:non}
|\hat{Q}_j^k|\leq C_{\varepsilon}\left(|e_j^{k+1}|+|e_j^k|\right),\, |\delta_x^+\hat{Q}_j^k|\leq C_{\varepsilon}\left(|\delta_x^+e_j^{k+1}|+|\delta_x^+e_j^k|+|e_j^{k+1}|+|e_j^k|\right),\, j\in\mathcal{T}_M^{0,1},\, k\ge0.
\end{align}

\end{lemma}
\begin{proof}
Under the assumptions of Lemma \ref{Nonlinearity_1},  there exists a constant $C_0>0$ independent of $h$, $\tau$ and $k$ such that $\|U^k\|_\infty\leq C_0,|u(x_j,t_k)|\leq C_0$ ($j\in\mathcal{T}_M^0$, $k\ge0$).

Recalling \eqref{eq:tildeQ},  $\hat{Q}_j^k$ \eqref{eq:QI2}  ($k\ge0$) can be written as
\begin{align*}
\hat{Q}_j^k=&
\frac{1}{2}\int_0^1q(\theta|u(x_j,t_{k+1})|^2+(1-\theta)|u(x_j,t_{k})|^2)\,d\theta\, (u(x_j,t_{k+1})+u(x_j,t_{k}))\\
&-\frac{1}{2}\int_0^1q(\theta|U_j^{k+1}|^2+(1-\theta)|U_j^{k}|^2)\,d\theta\, (U_j^{k+1}+U_j^{k})
\\
=&\frac{1}{2}\int_0^1q(\theta |u(x_j,t_{k+1})|^2+(1-\theta)|u(x_j,t_{k})|^2)\,d\theta \left(e_j^{k+1}+e_j^k\right)\\
&+\frac{1}{2}\int_0^1\left(q(\theta |u(x_j,t_{k+1})|^2+(1-\theta)|u(x_j,t_{k})|^2)-q(\theta|U_j^{k+1}|^2+(1-\theta)|U_j^{k}|^2)\right)\,d\theta\left(U_j^{k+1}+U_j^{k}\right).
\end{align*}
When $q(s)=\ln(\eps+\sqrt{s})$ \eqref{eq:tildeQ}. Combing the observations above and the $l^\infty$ bounds of $U^k$ and $u(\cdot,t_k)$, the first estimates in Lemma \ref{Nonlinearity_1} can be get by the similar procedure in Lemma \ref{Truncation_error_2}.

For the bound on $\delta_x^+\hat{Q}_j^k$, (see Appendix \eqref{appendix_C} for details) The detail is omitted here for brevity.
\end{proof}
Now, we present the error analysis.
\begin{proof}{\bf{Proof of \eqref{convergence_2} in Theorem \ref{convergence_fdm1}:}}
Computing the inner product of \eqref{err2_1} with $e^{k+1/2}\in Z_h^0$, and taking the imaginary part, 
using  Lemmas \ref{delta_x} and \ref{Nonlinearity_1}, we obtain
\begin{align}\label{eq:err-2:1}
\|e^{k+1}\|^2-\|e^k\|^2=&2\mu\tau \text{Im}(\hat{Q}(e^k,e^{k+1}),e^{k+1/2})+2\tau\text{Im}(R^{k},e^{k+1/2})\\
\lesssim&C_{\eps}^2\tau\left(\|e^{k+1}\|^2+\|e^k\|^2+|e_M^{k+1}|^2+|e_M^k|^2\right)+C\tau( C_{\eps}\tau^2+h^2)^2,\nonumber
\end{align}
where we have used the Cauchy inequality to find that $|\text{Im}(R^{k},e^{k+1/2})|\leq h
\sum\limits_{j\in\mathcal{T}_M^1}|R_j^{k}|\cdot|e_j^{k+1/2}|+h|R_M^{k}|\cdot|e_M^{k+1/2}|\lesssim \|e^{k+1}\|^2+\|e^k\|^2 
+ |e_M^{k+1}|^2+|e_M^k|^2+h^2|R_M^k|^2+C_{\varepsilon}^2\tau^{4}+Ch^4$.
On the other hand, computing the inner product of \eqref{err2_1} with $e^{k+1}-e^{k}\in Z_h^0$, and taking the real part, we have 
\begin{align*}
\text{Im}(\delta_t^+e^{k},e^{k+1}-e^k)
&=-\text{Re}(\delta_{x}^2e^{k+1/2},e^{k+1}-e^k)+ \lambda \text{Re}(b\circ e^{k+1/2},e^{k+1}-e^k)\\
&+\mu \text{Re}(\hat{Q}^k,e^{k+1}-e^k)+\text{Re}(R^{k},e^{k+1}-e^k),
\end{align*}
where $b\circ e^{k+1/2}\in Z_h^0$ means the component-wise multiplication, i.e.  $(b\circ e^{k+1/2})_j=b_je^{k+1/2}_j$.
Using  Lemma \ref{delta_x} and recalling $\delta_t^+e^{k+1/2}=-i(-\delta_{x}^2e^{k+1/2}+\lambda b\circ e^{k+1/2}+\mu\hat{Q}_j^k+R^{k})$, we get
\begin{align}\label{eq:ep1}
&\|\delta_x^+ e^{k+1}\|^2-\|\delta_x^+ e^k\|^2 +\lambda( |e_M^{k+1}|^2-|e_M^k|^2)\\
&=2\tau \text{Im}\left(\mu\hat{Q}^k, -\delta_{x}^2e^{k+1/2}+\lambda b\circ e^{k+1/2}+\mu\hat{Q}^k+R^{k}  \right)
-2\tau\text{Re}(R^k,\delta_t^+e^{k+1/2}).\nonumber
\end{align}
Following the proof of Theorem \ref{convergence_fdm1}, using  Lemmas \ref{delta_x} and \ref{Nonlinearity_1}, we can
derive
\begin{align*}
&\text{Im}\left(\mu\hat{Q}^k,-\delta_x^2e^{k+1/2}\right)\lesssim  C_{\eps}^2(\|\delta_x^+e^{k+1}\|^2+\|\delta_x^+e^k\|^2),\quad\text{Im}(\mu\hat{Q}^k, \lambda b\circ e^{k+1/2})
\lesssim C_{\eps}^2(|e_M^k|^2+|e_M^{k+1}|^2),\\
& \left|\left(\mu\hat{Q}^k,R^{k}   \right)\right|
\lesssim h\sum\limits_{j\in\mathcal{T}_M^1}|\hat{Q}_j^k|\cdot|R_j^k|+h|R_M^k|\cdot|\hat{Q}_M^k|
M_1M_1\lesssim C_{\eps}^2\left( \|e^k\|^2+\|e^{k+1}\|^2+|e_M^k|^2+|e_M^{k+1}|^2 \right)\\
&\quad\quad\quad\quad\quad\quad+C(C_{\eps}\tau^2++h^2)^2.
\end{align*}
Introducing
$S^k=\|\delta_x e^k\|^2+(2+4\lambda^2)\|e^k\|^2+\lambda|e_M^{k}|^2$ ($k\ge0$),  by the discrete Sobolev inequality,  we have 
\begin{align}\label{eq:Sk2}
&S^k\ge\frac{1}{2}\|\delta_x e^k\|^2+\|e^k\|^2+\frac{1}{2}|e_M^k|^2\quad k\ge0.
\end{align}
$(2+4\lambda^2)\eqref{eq:err-2:1} +\eqref{eq:ep1}$ leads to
\begin{equation}\label{eq:enrg}
S^{k+1}-S^k\leq -2\tau\text{Re}(R^k,\delta_t^+e^{k})
+C_{\eps}^2\tau (S^k+S^{k+1})+C\tau(C_{\eps}\tau^2+h^2)^2,\quad k\ge0.
\end{equation}
On the other hand, since $e^0=0$, summation-by-parts in time  (similar to the proof for CNFD) leads to
\begin{align}
-2\tau\text{Re}\sum\limits_{l=0}^k(R^l,\delta_t^+e^l)=&-2\,\text{Re}\left(  
(R^{k},e^{k+1})-(R^0,e^0)-\tau\sum_{l=1}^{k}(\delta_t^+R^{l-1},e^l)\right))\nonumber\\
\leq& 2h\left(\sum\limits_{j\in\mathcal{T}_M^1}(|R_j^k|\cdot|e_j^{k+1}|+|R_M^k|\cdot|e_M^{k+1}|\right)\nonumber\\
&\,+2h\tau
\left(\sum\limits_{l=1}^k\sum\limits_{j\in\mathcal{T}_M^1}|\delta_t^+R_j^{l-1}|\cdot|e_j^{l}|+|\delta_t^+R_M^{l-1}|\cdot|e_M^{l}|\right)\nonumber\\
\leq&\, \eta (\|e^{k+1}\|^2+|e_M^{k+1}|^2)+C_{\eps}^2\tau\sum\limits_{l=1}^k\left(\|e^l\|^2+|e_M^l|^2\right)+C_\eta(C_{\eps}\tau^2+h^2)^2,\label{eq:summ}
\end{align}
where $\eta>0$ can be chosen arbitrary and we have used the Cauchy inequality and the fact that
$h\sum\limits_{j\in\mathcal{T}_M^1}|\delta_t^+R_j^{l}|^2+h^2|R_M^l|^2\lesssim C(C_{\eps}\tau^2+h^2)^2$ (Lemma \ref{Truncation_error_2}).
Summing \eqref{eq:enrg} together for indices $0,1,\ldots,k$, combining \eqref{eq:summ}  and \eqref{eq:Sk2}, we can derive
\begin{align*}
S^{k+1}\lesssim C_{\eps}^2\tau\sum_{l=1}^{k+1}S^l+C(C_{\eps}\tau^2+h^2)^2,\quad k\ge0.
\end{align*}
Using the Gronwall inequality, for sufficiently small $
\tau$, we can get $S^{k+1}\lesssim C(C_{\eps}\tau^2+h^2)^2\lesssim e^{TC_{\eps}^2}(C_{\eps}\tau^2+h^2)^2$ and the estimates \eqref{convergence_2} follow.
\end{proof}

\section{Numerical experiments}
\label{sec:experiments}
In this section, we present numerical tests about accuracy and efficiency of the proposed finite difference schemes (CNFD) and perform some experiments concerning the conservative properties and convergence rate of the regularized model \eqref{RLSE}.
For numerical purpose, we  focus on the  equations \eqref{problem_1d}, and the numerical methods are programmed by using Matlab(The nonlinear procedure in algorithms is solved by the fixed point iteration).
\subsection{Spatial/temporal resolution} We firstly test the temporal and spatial discretization error of the proposed methods for CNFD. The initial data  for the problem \eqref{problem_1d} is considered, i.e.
the set of soliton initial data ($\lambda=1,\mu=-1$)
\begin{eqnarray} \label{ini:1}
	\begin{aligned}
		&u_0(x)=Q_{\omega}(x):=e^{\frac{\omega+1}{2}}e^{-\frac{1}{2}(\abs{x}-\frac{\lambda}{2})^2},
	\end{aligned}
\end{eqnarray}
where $\omega=1$ and the exact solution reads $u(x,t)=e^{i\omega t}Q_{\omega}(x)$.
The computational domain is set as $\Omega\times [0,T]$ ($\Omega=[-12,12]$, $T=1$) with the homogeneous Dirichlet boundary conditions. The discretization is prescribed in section \ref{sec:fds} and \eqref{eq:cnfd}

Denote $u_{h,\tau}^{n}$ as the numerical solution obtained by the proposed numerical method with mesh size $h$ and time step $\tau$ and $e_j^n=(u_{h,\tau})_{j}^{n}-u(x_j,t_n)$. In order to quantify the numerical errors, we introduce the $L^2,L^{\infty}$ and $H^1$ errors at time $t_n$ as
\begin{equation*}
	e_{h,\tau}(t_n)=\|e^n\|,\quad
	e_{\infty,h,\tau}(t_n)=\|e^n\|_{\infty},\quad
	e_{x,h,\tau}(t_n)=\sqrt{\|e^n\|^2+\|\delta_x^+ e^n\|^2}.
\end{equation*}

Here, the exact solution $(u(h,\tau))_j^n$ is obtained numerically by CNFD \eqref{eq:cnfd} with a very small time step $\tau_e$ and very small mesh size $h_e$ under corresponding $\varepsilon$. 

\begin{table}[!htbp]
	{\footnotesize
		\caption{Spatial error of $e_{h,\tau}(T)$ at $T=1$ under $\tau_e=0.1/2^{12},h_e=1/2^{12}$ with initial value \ref{ini:1}.}\label{tab:spatial1}
		\begin{center}
			\begin{tabular}{|c|c|c|c|c|c|c|c|}
				\hline
				$e_{h,\tau}(T)$ &$h_0=0.5$ &$h_0/2$ &$h_0/2^2$& $h_0/2^3$& $h_0/2^4$ & $h_0/2^5$& $h_0/2^6$\\
				\hline
				$\varepsilon=0.0025$&       8.28E-02 &  2.05E-2  & 5.12E-3 & 1.28E-3& 3.20E-4 &7.99E-5 &2.00E-5   \\
				order&-&2.01&2.00&2.00&2.00&2.00&2.00\\
				$\varepsilon/4$&     8.31E-2   &  2.06E-2  & 5.14E-3 & 1.28E-3& 3.21E-4 &8.01E-5 &2.00E-5     \\
				order&-&2.01&2.00&2.00&2.00&2.00&2.00\\
				$\varepsilon/4^2$&       8.32E-2     &    2.06E-2    & 5.14E-3 & 1.28E-3& 3.21E-4 &8.02E-5 &2.00E-05  \\
				order&-&2.01&2.00&2.00&2.00&2.00&2.00\\
				$\varepsilon/4^3$&       8.33E-2    &  2.06E-2     & 5.14E-3 & 1.28E-3& 3.21E-4  &8.02E-5 & 2.00E-5  \\
				order&-&2.01&2.00&2.00&2.00&2.00&2.00\\
				\hline
			\end{tabular}
		\end{center}
	}
\end{table}

\begin{table}[!htbp]
	{\footnotesize
		\caption{Spatial error of $e_{x,h,\tau}(T)$ at $T=1$ under $\tau_e=0.1/2^{12},h_e=1/2^{12}$ with initial value \ref{ini:1}.}\label{tab:spatial2}
		\begin{center}
			\begin{tabular}{|c|c|c|c|c|c|c|c|}
				\hline
				$e_{x,h,\tau}(T)$ &$h_0=0.5$ &$h_0/2$ &$h_0/2^2$& $h_0/2^3$& $h_0/2^4$ & $h_0/2^5$& $h_0/2^6$\\
				\hline
				$\varepsilon=0.0025$&      2.00E-1 &  5.02E-2  & 1.28E-2 & 3.17E-3& 7.93E-4 &1.98E-4 &4.96E-5   \\
				order&-&1.99&1.98&2.01&2.00&2.00&2.00\\
				$\varepsilon/4$&    2.01E-1    &  5.03E-2  & 1.28E-2 & 3.18E-3& 7.95E-4 &2.00E-4  &5.01E-05     \\
				order&-&2.00&1.98&2.01&2.00&1.99&2.00\\
				$\varepsilon/4^2$&       2.02E-1    &  5.04E-2     & 1.28E-2 & 3.21E-3& 8.01E-4 &2.00E-4 &4.98E-5  \\
				order&-&2.00&1.98&1.99&2.00&2.00&2.01\\
				$\varepsilon/4^3$& 2.02E-1    &  5.04E-2     & 1.28E-2 & 3.21E-3& 7.99E-4 &1.99E-4 &4.95E-5  \\
				order&-&2.00&1.98&1.99&2.00&2.01&2.00\\
				\hline
			\end{tabular}
		\end{center}
	}
\end{table}
\begin{table}[!htbp]
	{\footnotesize
		\caption{Spatial error of $e_{\infty,h,\tau}(T)$ at $T=1$ under $\tau_e=0.1/2^{12},h_e=1/2^{12}$ with initial value \ref{ini:1}.}\label{tab:spatial3}
		\begin{center}
			\begin{tabular}{|c|c|c|c|c|c|c|c|}
				\hline
				$e_{\infty,h,\tau}(T)$ &$h_0=0.5$ &$h_0/2$ &$h_0/2^2$& $h_0/2^3$& $h_0/2^4$ & $h_0/2^5$& $h_0/2^6$\\
				\hline
				$\varepsilon=0.0025$&       5.06E-2 &  1.28E-2  & 3.23E-3 & 8.11E-4& 2.04E-4  &5.11E-5 &1.28E-5   \\
				order&-&1.98&1.99&1.99&1.99&2.00&2.00\\
				$\varepsilon/4$&     5.07E-2    &  1.28E-2  & 3.23E-3 & 8.12E-4& 2.02E-4 &5.09E-5 &1.27E-5     \\
				order&-&1.98&1.99&1.99&2.01&1.99&2.00\\
				$\varepsilon/4^2$&       5.07E-2    &  1.28E-2     & 3.22E-3 & 8.12E-4& 2.04E-4 &5.12E-5 &1.28E-5  \\
				order&-&1.98&1.99&1.99&1.99&2.00&2.00\\
				$\varepsilon/4^3$&   5.07E-2   &  1.28E-2     & 3.22E-3  & 8.18E-4& 2.06E-4 &5.09E-5 &1.28E-5  \\
				order&-&1.98&1.99&1.98&1.99&2.02&2.00\\
				\hline
			\end{tabular}
		\end{center}
	}
\end{table}

\begin{table}[!htbp]
	{\footnotesize
		\caption{Temporal error of $e_{h,\tau}(T)$ at  $T=1$ under $\tau_e=0.1/2^{12},h_e=1/2^{12}$ with initial value \ref{ini:1}.}\label{tab:temporal1}
		\begin{center}
			\begin{tabular}{|c|c|c|c|c|c|c|c|c|}
				\hline
				$e_{h,\tau}(T)$ &$\tau_0=0.05$ &$\tau_0/2$ &$\tau_0/2^2$& $\tau_0/2^3$ &$\tau_0/2^4$ &$\tau_0/2^5$&$\tau_0/2^6$&$\tau_0/2^7$\\
				\hline
				$\varepsilon=0.0125$&       1.00E-3  & 2.53E-4 &  6.43E-5  & 1.68E-5 &  4.77E-6  & 1.57E-6 & 5.97E-7  & 2.51E-7   \\
				order&-&1.99 & 1.98 &  1.93&  1.82 &  1.61 &  1.39&  1.25\\
				$\varepsilon/2$&  9.64E-4 &  2.42E-4  & 6.08E-5 & 1.54E-5 &4.02E-6  &1.13E-6 &3.59E-7&1.36E-7    \\
				order&-&1.99&1.99&1.98&1.94&1.84&1.65&1.41\\
				
				$\varepsilon/2^2$&  9.45E-4 &  2.37E-4 &  5.93E-5  & 1.49E-5 &  3.76E-6  & 9.75E-7 & 2.66E-7 &  8.41E-8   \\
				order&-&1.99&   2.00&   1.99&   1.98&   1.95&   1.87&  1.66\\

				$\varepsilon/2^3$&       9.37E-4 &  2.34E-4 &  5.86E-5 &  1.47E-5 &  3.68E-6  & 9.30E-7 &2.36E-7  & 6.50E-8  \\
				order&-&2.00 &  2.00 &2.00&  2.00  & 1.98&   1.98& 1.86\\

				$\varepsilon/2^4$& 9.33E-4 &  2.33E-4 &  5.83E-5 &  1.46E-5 &  3.65E-6 &  9.16E-7 &2.27E-7  & 5.92E-8    \\
				order&-&2.00&  2.00 &  2.00&   2.00&   1.99&   2.01&  1.94\\
				
				$\varepsilon/2^5$&       9.31E-4  & 2.33E-4  & 5.82E-5 &  1.46E-5 &  3.64E-6 &  9.11E-7 & 2.25E-7 &  5.76E-8   \\
				order&-&2.00&  2.00&  2.00&   2.00&   2.00&   2.02& 1.97\\
				$\varepsilon/2^6$& 9.30E-4  & 2.33E-4  & 5.82E-5 &  1.45E-5 &  3.63E-6 &  9.10E-7 & 2.24E-7 &  5.72E-8  \\
				order&-&2.00 &  2.00 &  2.00 & 2.00&   2.00&   2.02& 1.97\\
				$\varepsilon/2^7$&       9.30E-4 &  2.32E-4 &  5.81E-5 &  1.45E-5 &  3.63E-6 &  9.09E-7 & 2.24E-7 &  5.71E-8    \\
				order&-&2.00&2.00&2.00&2.00&2.00&2.02&1.97\\
				\hline
			\end{tabular}
		\end{center}
	}
\end{table}
\begin{table}[!htbp]
	{\footnotesize
		\caption{Temporal error of $e_{x,h,\tau}(T)$ at  $T=1$ under $\tau_e=0.1/2^{12},h_e=1/2^{12}$ with initial value \ref{ini:1}.}\label{tab:temporal2}
		\begin{center}
			\begin{tabular}{|c|c|c|c|c|c|c|c|c|}
				\hline
				$e_{x,h,\tau}(T)$ &$\tau_0=0.05$ &$\tau_0/2$ &$\tau_0/2^2$& $\tau_0/2^3$ &$\tau_0/2^4$ &$\tau_0/2^5$&$\tau_0/2^6$&$\tau_0/2^7$\\
				\hline
				$\varepsilon=0.0125$&      1.35E-3 &  \bf{3.72E-4} &  1.23E-4 &  5.52E-5 &  3.04E-5 &  1.70E-5 & 9.54E-6 &  5.40E-6  \\
				order&-&\bf{1.86} &  1.60 &  1.15&   0.86  & 0.84  & 0.83 & 0.82\\
				
				$\varepsilon/2$&  1.20E-3  & 3.14E-4 &  \bf{8.85E-5} &  3.18E-5 &  1.57E-5 &  8.54E-6 & 4.77E-6 &  2.70E-6    \\
				order&-&1.94 & \bf{1.83} &  1.48&   1.02&   0.88  & 0.84  & 0.82\\
				
				$\varepsilon/2^2$&  1.13E-3 &  2.87E-4 &  7.48E-5 & \bf{ 2.19E-5}  & 8.66E-6  & 4.36E-6 & 2.39E-6 &  1.35E-6  \\
				order&-& 1.97&  1.94&  \bf{ 1.77}&   1.34&   0.99  & 0.87 & 0.83\\

				$\varepsilon/2^3$&       1.10E-3 &  2.76E-4 &  6.98E-5 &  1.83E-5 &  \bf{5.64E-6} &  2.37E-6 & 1.21E-6 &  6.72E-7 \\
				order&-& 2.00 &  1.98&  1.93&   \bf{1.70}&   1.25&   0.96&  0.85\\

				$\varepsilon/2^4$&   1.08E-3 &  2.72E-4&   6.81E-5 &  1.73E-5 &  4.61E-6 &  \bf{1.50E-6} & 6.43E-7 &  3.36E-7   \\
				order&-&2.00 & 1.99&  1.98&   1.90&   \bf{1.62}&   1.22&  0.94\\
				
				$\varepsilon/2^5$&       1.08E-3 &  2.70E-4 &  6.76E-5 &  1.70E-5 &  4.32E-6 &  1.18E-6 &\bf{3.89E-7} &  1.73E-7  \\
				order&-&2.00&   2.00 &  1.99&   1.97&   1.87&  \bf{1.60}& 1.17\\
				$\varepsilon/2^6$& 1.08E-3 &  2.70E-4 &  6.75E-5 &  1.69E-5 &  4.25E-6 &  1.09E-6 & 2.95E-7 &  \bf{1.01E-7} \\
				order&-&2.00&  2.00&   2.00&   1.99&   1.97&   1.88&  \bf{1.55}\\
				$\varepsilon/2^7$& 1.08E-3 &  2.70E-4 &  6.75E-5 &  1.69E-5 &  4.23E-6 &  1.06E-6 & 2.68E-7 &  7.47E-8    \\
				order&-&2.00&  2.00&  2.00 &  2.00&   1.99&   1.99&1.84\\
				\hline
			\end{tabular}
		\end{center}
	}
\end{table}

\begin{table}[!htbp]
	{\footnotesize
		\caption{Temporal error of $e_{\infty,h,\tau}(T)$ at  $T=1$ under $\tau_e=0.1/2^{12},h_e=1/2^{12}$ with initial value \ref{ini:1}.}\label{tab:temporal3}
		\begin{center}
			\begin{tabular}{|c|c|c|c|c|c|c|c|c|}
				\hline
				$e_{\infty,h,\tau}(T)$ &$\tau_0=0.05$ &$\tau_0/2$ &$\tau_0/2^2$& $\tau_0/2^3$ &$\tau_0/2^4$ &$\tau_0/2^5$&$\tau_0/2^6$&$\tau_0/2^7$\\
				\hline
				$\varepsilon=0.0125$&  6.23E-4 &  1.56E-4 &  3.95E-5 &  1.02E-5 &  3.44E-6 &  7.44E-7 & 2.43E-7 &  1.15E-7  \\
				order&-&1.99&  1.99&   1.95&   1.58&   2.21&   1.61&  1.09\\
				
				$\varepsilon/2$&  6.06E-4 &  1.52E-4 &  3.79E-5 &  9.49E-6 &  2.50E-6 &  7.36E-7 & 1.97E-7 &  8.09E-8    \\
				order&-&2.00&   2.00&   2.00&  1.93&   1.76&   1.90&  1.29\\
				
				$\varepsilon/2^2$& 5.90E-4&   1.48E-4&   3.69E-5&   9.18E-6 &  2.42E-6 &  6.51E-7 & 1.62E-7 &  5.34E-8 \\
				order&-&2.00&   2.00&  2.01&   1.93&   1.89&  2.00&  1.60\\

				$\varepsilon/2^3$& 5.78E-4  & 1.45E-4  & 3.62E-5  & 9.02E-6 &  2.36E-6 &  5.75E-7&  1.44E-7 &  4.09E-8 \\
				order&-& 2.00&   2.00&   2.00&   1.93&   2.04&   2.00&  1.81\\

				$\varepsilon/2^4$&  5.71E-4 &  1.43E-4 &  3.57E-5 &  8.93E-6 &  2.26E-6 &  5.73E-7 & 1.44E-7 &  4.10E-8  \\
				order&-&2.00&  2.00&   2.00&   1.98&   1.98&  1.99&  1.81\\
				
				$\varepsilon/2^5$& 5.67E-4  & 1.42E-4 &  3.54E-5&   8.86E-6 &  2.21E-6 &  5.63E-7 & 1.41E-7 &  3.78E-8  \\
				order&-&2.00 & 2.00&   2.00&   2.00&   1.98&   2.00&  1.89\\
				$\varepsilon/2^6$& 5.65E-4&   1.41E-4 &  3.53E-5 &  8.84E-6&   2.20E-6 &  5.53E-7&  1.38E-7 &  3.64E-8 \\
				order&-& 2.00&   2.00&   2.00&   2.00&   2.00&   2.00&  1.92\\
				$\varepsilon/2^7$& 5.65E-4 &  1.41E-4 &  3.53E-5 &  8.83E-6 &  2.20E-6 &  5.53E-7&1.39E-7 &  3.61E-8    \\
				order&-&2.00&   2.00&   2.00&   2.00&   1.99&   1.99& 1.94\\
				\hline
			\end{tabular}
		\end{center}
	}
\end{table}

For the spatial error analysis, we set $\tau=\tau_e$ small enough such that the temporal error can be neglected. Tables \ref{tab:spatial1}-\ref{tab:spatial3} list the spatial errors of $L^2$, $H^1$ and the $L^{\infty}$ discrete norm for system \eqref{problem_1d} with initial data \eqref{ini:1}. For temporal error analysis, the mesh size $h=h_e$ is very small such that the spatial discretization error is negligible. Tables \ref{tab:temporal1}-\ref{tab:temporal3} list the temporal errors of $L^2$, $H^1$ and the $L^{\infty}$ discrete norm for the problem \eqref{problem_1d} with initial data \eqref{ini:1}. 

From the Tables \ref{tab:spatial1}-\ref{tab:spatial3} and Tables \ref{tab:temporal1}-\ref{tab:temporal3}, we can observe that:
\begin{itemize}
	\item[(1)] As shown in the Tables \ref{tab:spatial1}-\ref{tab:spatial3},  second-order accuracy in space is clear for CNFD.
	\item[(2)] From the Tables \ref{tab:temporal1}-\ref{tab:temporal3}, we observe that CNFD \eqref{Fdm_2-1}-\eqref{Fdm_2-2} converges quadratically at $O(\tau^2)$ when $\varepsilon$ is sufficiently small, e.g., $\varepsilon\leq C\tau$(cf. lower triangle below the diagonal in bold letter in Table \ref{tab:temporal1}-\ref{tab:temporal3}, especially in Table \ref{tab:temporal2}).
\end{itemize}

\subsection{Conservation laws}
The conservation characteristics of the system \eqref{problem_1d} are very important,
when constructing numerical methods. Here, the conservative characteristic of our numerical method (CNFD)
will be considered in this subsection. Under the condition $T=10,\tau=0.01, h=0.01$ with the initial value condition \eqref{ini:1}, the Figure \ref{fig:s} shows the conservation
laws ((a)(c) for mass,(b)(d) for energy) with different $\varepsilon$. From the Figure \ref{fig:s}, we
can obtain that our numerical method is conservative method.
\begin{figure}[htbp]
	\centering
	\subfloat[]{\includegraphics[width=2.7in,height=1.7in]{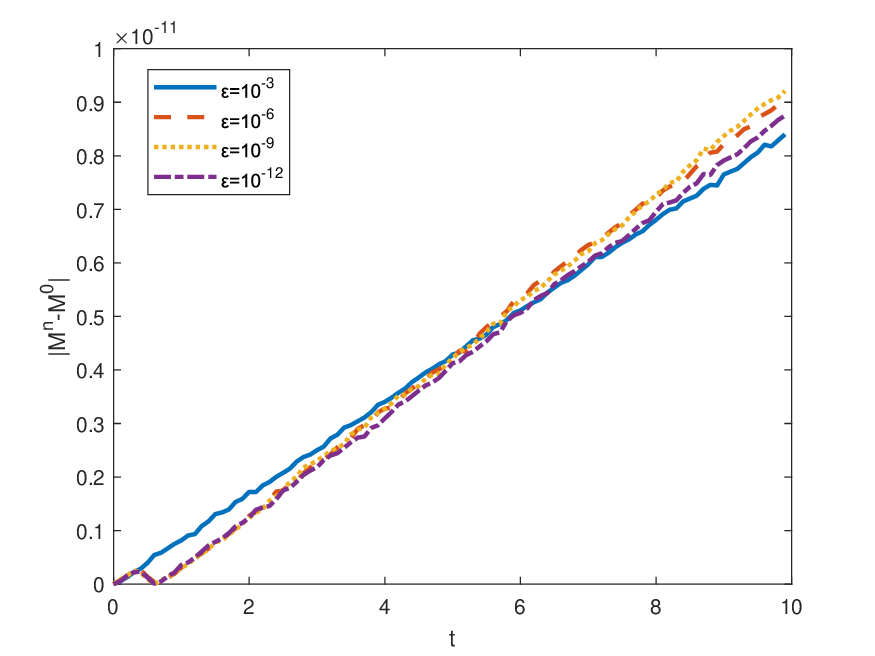}}
	\subfloat[]{\includegraphics[width=2.7in,height=1.7in]{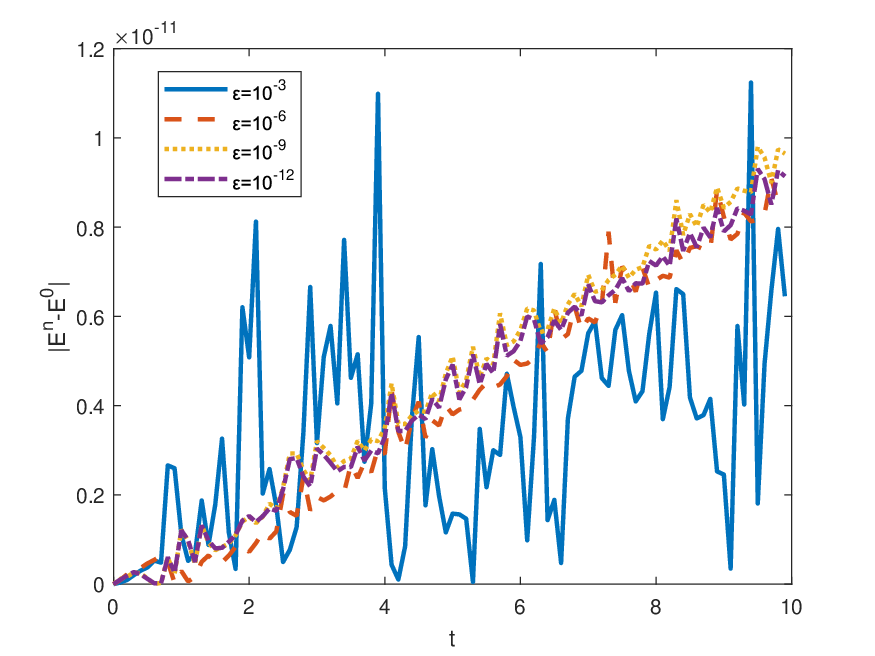}}
	
	\subfloat[]{\includegraphics[width=2.7in,height=1.7in]{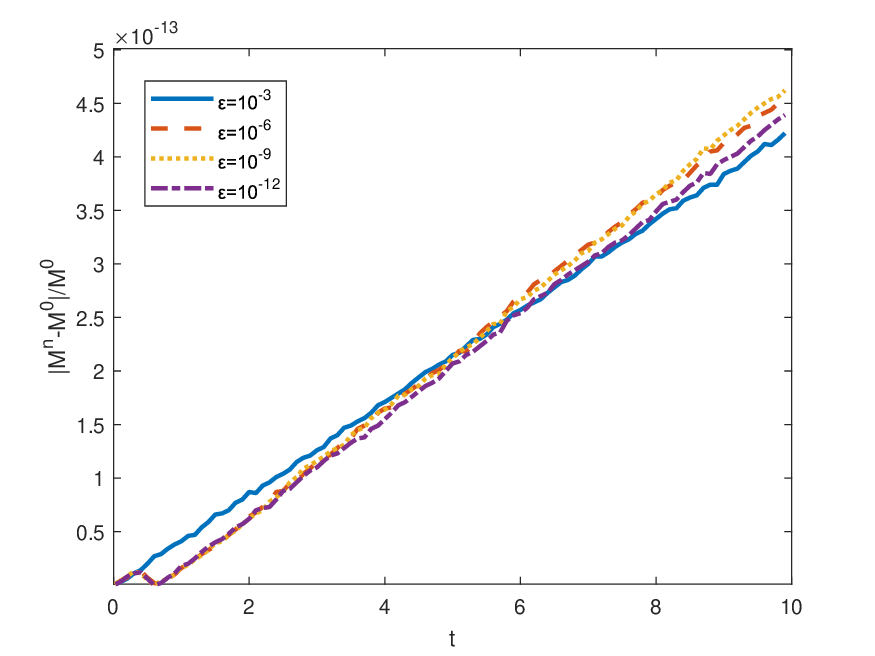}}
	\subfloat[]{\includegraphics[width=2.7in,height=1.7in]{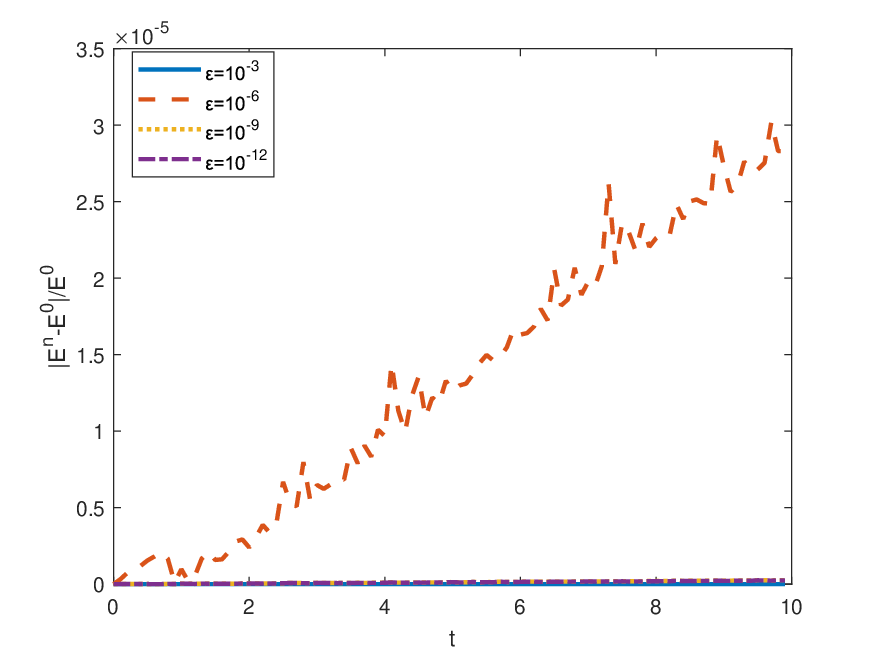}}
	\caption{ The conservation laws of CNFD with initial value \eqref{ini:1} for different $\varepsilon$.}
	\label{fig:s}
\end{figure}

\subsection{ Convergence rate of the regularized model}
In the subsection, we consider the error between the solution of \eqref{problem} and the solution of \eqref{RLSE}. The exact solution of the \eqref{problem} is chosen as $u(x,t)=e^{i\omega t}Q_{\omega}(x) \eqref{ini:1}$ and the exact solution of \eqref{RLSE} is obtained numerically by CNFD \eqref{eq:cnfd} with a very small time step $\tau_e=0.1\times\frac{1}{2^{12}}$ and very small mesh size $h_e=\frac{1}{2^{12}}$. Figure \ref{comver} shows $\|e^n\|_{L^2},\|e^n\|_{H^1},\|e^n\|_{\infty}$ at $T=1$.
From Figure \ref{comver} , we can draw the following conclusion: The solution of \eqref{RLSE} converges linearly to that of \eqref{problem} in terms of the regularization parameter $\varepsilon$ in all $L^2$-norm, $L^{\infty}$-norm and $H^1$-norm. 
\begin{figure}[htbp]
	\centering
	\subfloat[]{\includegraphics[width=2.7in,height=1.7in]{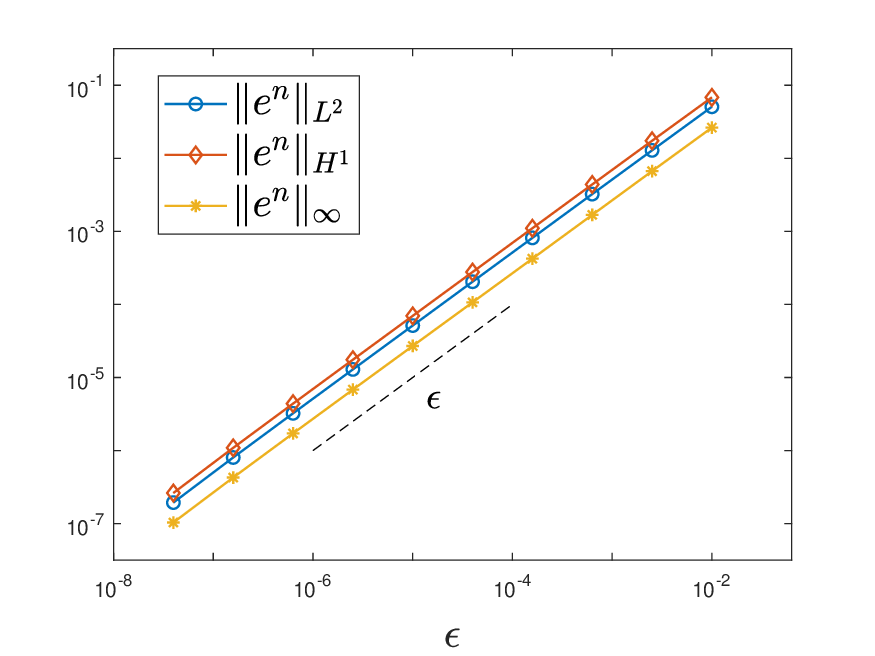}}
	\caption{Convergence rate of the regularized model for different $\epsilon$.}
	\label{comver}
\end{figure}

\subsection{Stability test} 
The purpose of this subsection is to study the orbital stabilities of standing wave solutions $e^{i \omega t} \phi_\omega(x)$ to the problem \eqref{problem}, 
\begin{equation*}
\phi_\omega(x)=e^{\frac{\omega+1}{2}}e^{-\frac{1}{2}(\abs{x}-\frac{\lambda}{2})^2}.
\end{equation*}
First,  we recall the following definition of orbital stability.
\begin{definition}
We say that the standing wave solution $e^{i \omega t} \phi_\omega$ of (1.1) is orbitally stable in $W(\mathbb{R}):=\left\{u \in H^1(\mathbb{R}):|u|^2 \log |u|^2 \in L^1(\mathbb{R})\right\}$, if for any $\varepsilon>0$ there exists $\delta>0$ such that if $u_0 \in W(\mathbb{R})$ and $\left\|u_0-\phi_\omega\right\|_{W(\mathbb{R})}<\delta$, then the solution $u(t)$ of $(1.1)$ with $u(0)=u_0$ exists globally and satisfies
\begin{equation*}
\sup_{t \in \mathcal{R}} \inf_{\theta \in \mathbb{R}}\left\|u(t)-e^{i \theta} \phi_\omega\right\|_{W(\mathbb{R})}<\varepsilon .
\end{equation*}
Otherwise, $e^{i \omega t} \phi_\omega$ is said to be orbitally unstable.
\end{definition}

The following results on the orbital stability and instability have been proved:
\begin{thm}\label{thm:orbital}
\cite{Good:2004,Pava:2017}
Let $\lambda \neq 0,\mu=-1,\omega\in \mathbb{R}$ and $\varphi_{\omega}$ be defined by \eqref{ini:1}. Then the following assertions hold.
\begin{itemize}
\item[(i)] If $\lambda<0$, then the standing wave $e^{i \omega t} \varphi_{\omega}$ is orbitally stable in $W(\mathbb{R})$.
\item[(ii)] If $\lambda>0$, then the standing wave $e^{i \omega t} \varphi_{\omega}$ is orbitally unstable in $W(\mathbb{R})$.
\end{itemize}
\end{thm}

Now we apply the CNFD to numerically study the orbital stability of the standing waves of the system \eqref{problem_1d} with $\lambda\in \mathbb{R}$ and $\mu=-1$. In order to do so, we choose the initial data as
\begin{align}
&u_0(x)=Q_{\omega}(x)+\eta(1+0.5i)\sin(x),\label{perturbation_2}
\end{align}
where $\eta>0$ is a parameter to measure the amplitude of the perturbation. Then, the system \eqref{problem_1d} is solved numerically on $\Omega=[-40,40]$ by the CNFD method with  very fine mesh sizes $\tau=0.001$ and $2M=8000$. Figures \ref{fig:stab_m}-\ref{fig:stab_p} 
show the difference $\|u(t)-e^{i\omega t}\phi_{\omega}\|_{W(\mathbb{R})}$($e(\omega,\eta)$) and $\min_{\theta\in \mathcal{R}}{\|u(t)-e^{i\theta}\phi_{\omega}\|_{W(\mathbb{R})}}$($e_{\theta}(\omega,\eta)$) for two cases, i.e. $\lambda <0$ and $\lambda >0$ with different $\eta,\omega,T$. From Figures \ref{fig:stab_m}-\ref{fig:stab_p},
 we can draw the following observations:
\begin{itemize}
\item[(1)]
when $\lambda=-1(\lambda <0 ),\omega=1,T=30$, no matter what the value of the parameter $\eta$ and $T$ are, the standing wave solution is always orbital stable; and the norm $e_{\theta}(\omega,\eta)$ changes periodically (please see Figure \ref{fig:stab_m}).
\item[(2)]
when $\lambda=1(\lambda >0 ),\omega=1,T=50$ and $\eta$ decreases from $10^{-1}$ to $10^{-8}$  ($\eta=10^{-1},10^{-3},10^{-6},10^{-8}$), the standing wave solution is always orbital unstable, although the time of blow-up phenomenon is delayed  (please see Figure \ref{fig:stab_m}).

\end{itemize}

\begin{figure}[htb]
  \centering
  \subfloat[]{\includegraphics[width=2.7in,height=1.7in]{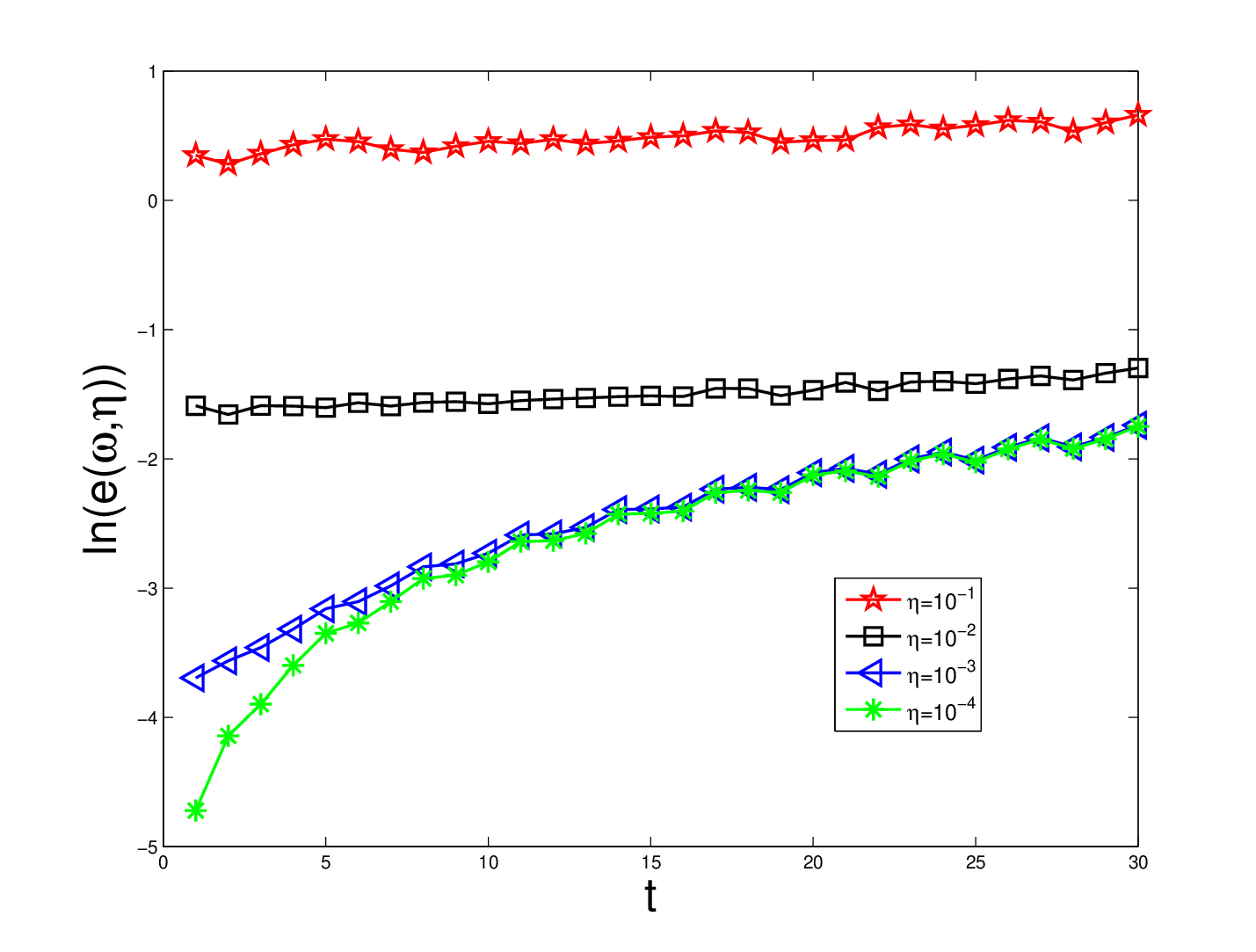}}
  \subfloat[]{\includegraphics[width=2.7in,height=1.7in]{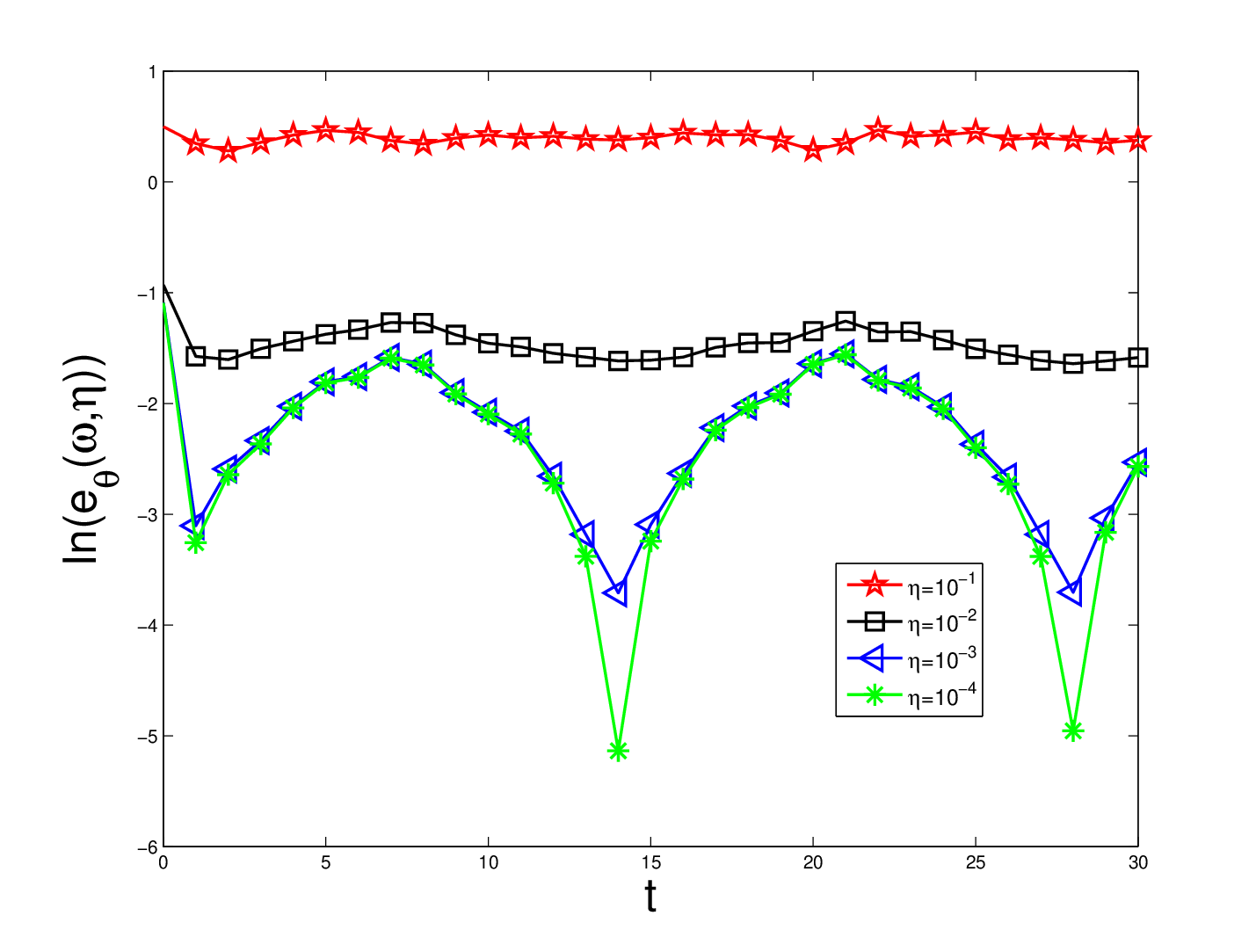}}
\caption{ The log-plot of $e(\omega,\eta)$ and $e_{\theta}(\omega,\eta)$ for different $\eta$ of the perturbation \eqref{perturbation_2} under $\lambda=-1,T=30$.}
\label{fig:stab_m}
\end{figure}
\begin{figure}[htb]
  \centering
  \subfloat[]{\includegraphics[width=2.7in,height=1.7in]{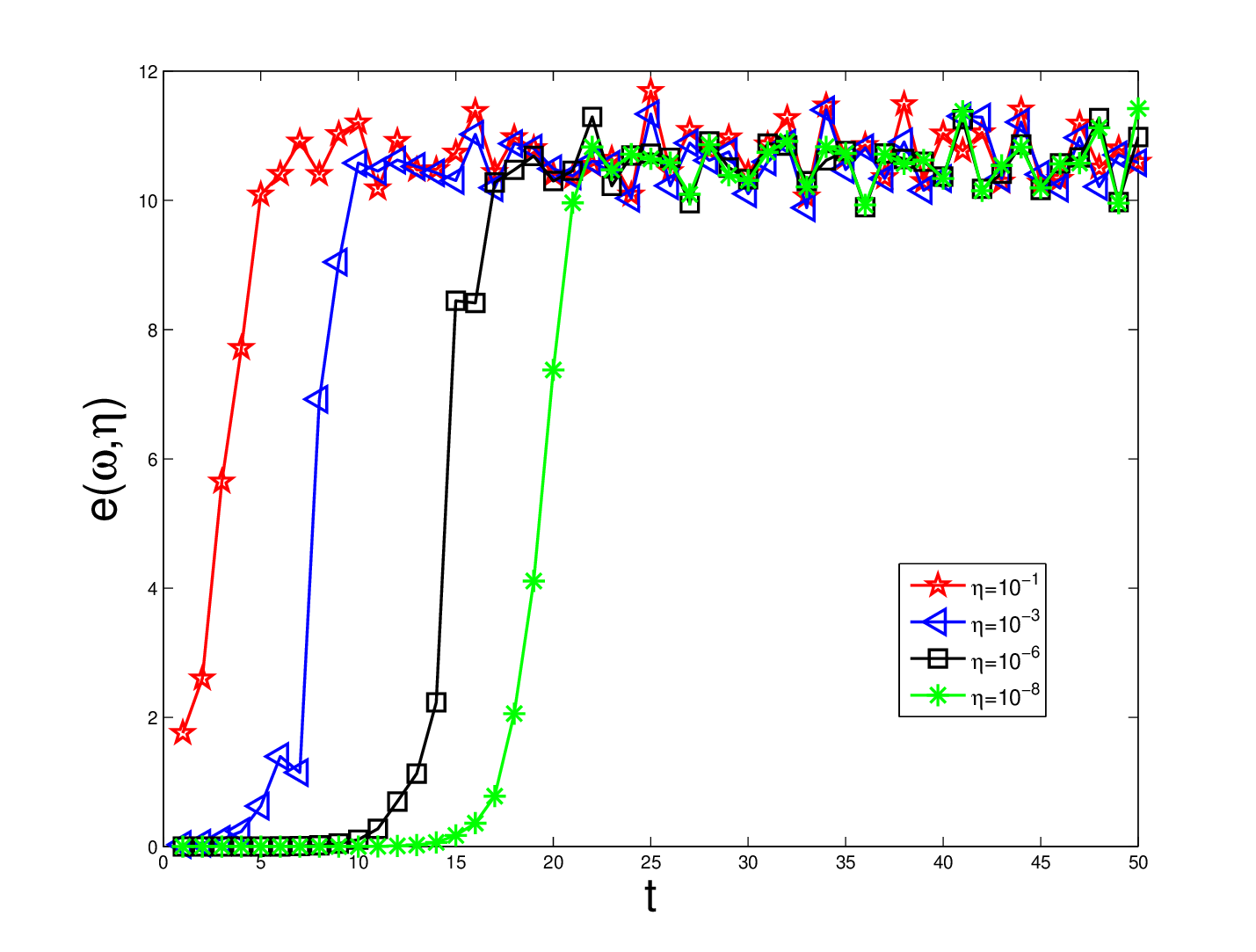}}
  \subfloat[]{\includegraphics[width=2.7in,height=1.7in]{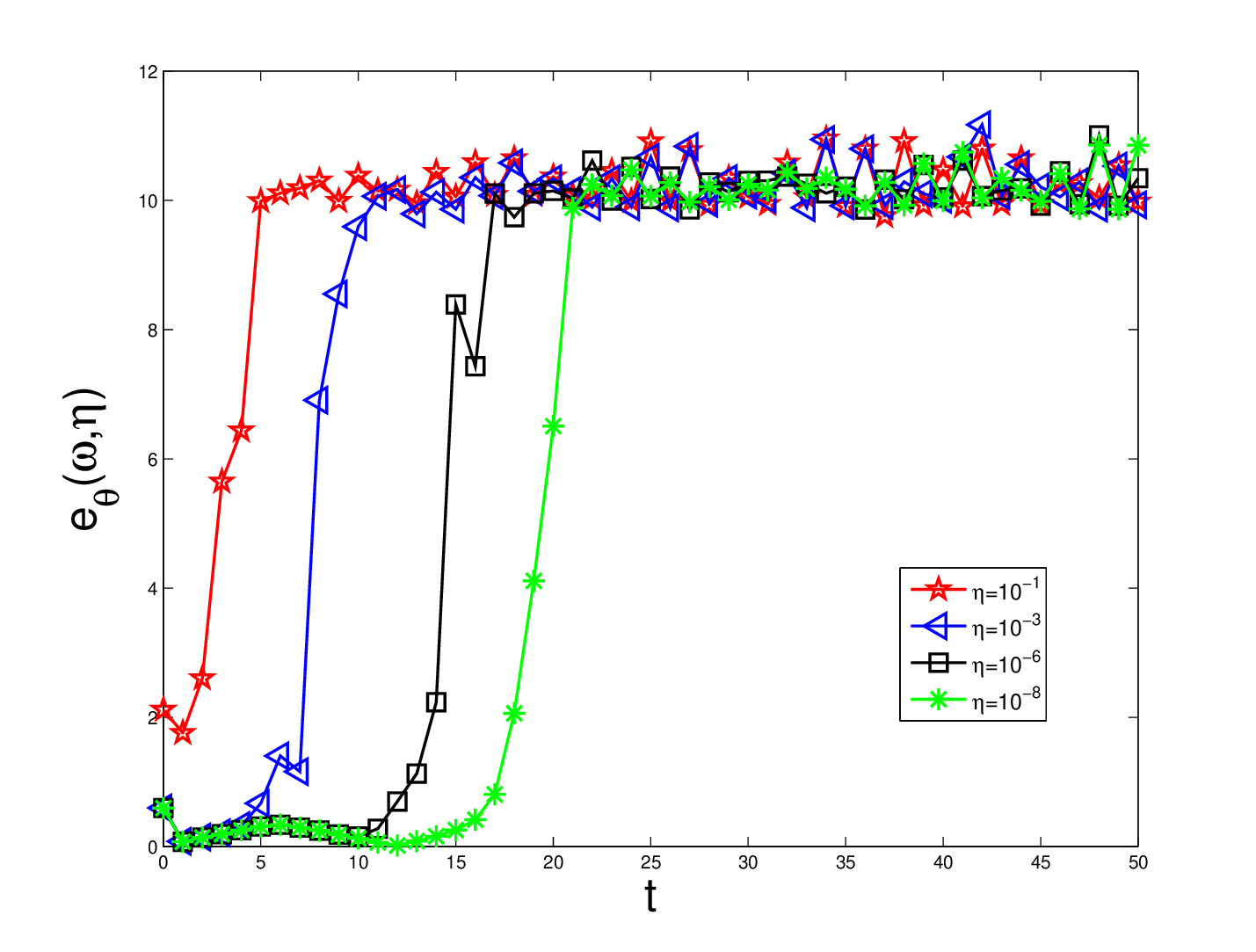}}
\caption{ The plot of $e(\omega,\eta)$ and $e_{\theta}(\omega,\eta)$ for different $\eta$ of the perturbation \eqref{perturbation_2} under $\lambda=1,T=50$.}
\label{fig:stab_p}
\end{figure}

\section{Conclusions}
\label{sec:conclusions}
In this paper, we analyzed a conservative finite difference schemes for  the Regularized Logarithmic Schr\"{o}dinger equation with a Dirac delta potential in 1D. The conservative properties and the $H^1$ error analysis of the finite difference methods were carried out by treating the interface condition carefully. The results implied that a simple approximation of the Dirac delta potential could lead to the second order accuracy in space.  Numerical results confirmed the convergence rates in time and space, discrete conservative laws and the convergence rate from RLSE to LSE of the proposed methods. Finally,  stability test verified the orbital stability results of the standing wave solution. In the future, we will consider how to do the regularization to improve the convergence order of the regularized model to the original problem.


\clearpage
\begin{appendices}
\appendix

\section{The {\it a priori} $l^\infty$ bounds of the numerical solutions of CNFD:}
\label{appendix:A}
We show the proof of CNFD here.
Firstly, applying the discrete Sobolev inequality and H\"older's inequality, we have
\begin{align*}
\mu|U_M^k|^2\leq |\mu|\|U^k\|^2_{\infty}
&\leq 4|\mu \|U^k\| \|\delta_x^+U^k\|
\leq  4|\mu|(2\eta\|\delta_x^+U^k\|^2+ \frac{1}{4\eta}\|U^k\|^2),\quad \forall \eta>0,
\end{align*}
and 
\begin{align*}
h\sum\limits_{j\in\mathcal{T}_M^0}Q(|U_j^{k}|^2)
&=|\lambda| h\sum\limits_{j\in\mathcal{T}_M^0}\left ( 2(|U_j^{k}|^2-\eps^2)\ln(|U_j^k|+\eps)-|U_j^k|^2+2\eps|U_j^k| \right )\\
&\leq |\lambda|\left\{ M^0+ 2\eps\|U^k\|_{\infty} + h\sum\limits_{j\in\mathcal{T}_M^0} 2(|U_j^{k}|^2-\eps^2)\ln(|U_j^k|+\eps) \right \},
\end{align*}
since $|\ln(x)|<\frac{1}{x},0<x\leq 1;\ln(x)<x,x \geq 1,$ thus if $0<|U_j^k|+\eps\leq 1$
\begin{align*}
h\sum\limits_{j\in\mathcal{T}_M^0}Q(|U_j^{k}|^2)
&\leq |\lambda|\left\{ M^0+ 2\eps\|U^k\|_{\infty} + h\sum\limits_{j\in\mathcal{T}_M^0} 2||U_j^{k}|-\eps| \right \}
\leq  |\lambda|\left\{ M^0+ 2\eps\|U^k\|_{\infty} + h\sum\limits_{j\in\mathcal{T}_M^0} 2(|U_j^{k}|+\eps) \right \}\\
&\leq |\lambda|\left\{ M^0+ 2(\eps+1)\|U^k\|_{\infty} + \eps \right \},
\end{align*}
else if $|U_j^k|+\eps\geq 1(|U_j^k|>\eps)$,
\begin{align*}
h\sum\limits_{j\in\mathcal{T}_M^0}Q(|U_j^{k}|^2)
&\leq |\lambda|\left\{ M^0+ 2\eps\|U^k\|_{\infty} + h\sum\limits_{j\in\mathcal{T}_M^0} 2(|U_j^{k}|^2-\eps^2)(|U_j^{k}|+\eps) \right \}\\
&\leq  |\lambda|\left\{ M^0+ 2\eps\|U^k\|_{\infty} + h\sum\limits_{j\in\mathcal{T}_M^0} 2(|U_j^{k}|^3-\eps^3-\eps^2|U_j^k|+\eps|U_j^k|^2) \right \}\\
&\leq |\lambda|\left\{ (\eps+1)M^0+ 2(\eps+1-\eps^2)\|U^k\|_{\infty} -2 \eps^3 + 2\|U_j^{k}\|_{3}^3\right \}
\end{align*}
\begin{align*}
2\lambda h\sum_{j=0}^{2M-1}|U_j^k|^3&
\leq 2|\lambda|\|U^k\|_\infty\cdot\|U^k\|^{2} \leq 4|\lambda|\|\delta_x^+U_k\|^{\frac{1}{2}}\|U^k\|^{\frac{5}{2}}
\\
&\leq  \left(\eps\|\delta_x^+U^k\|^2+\frac{C}{\eps}\|U^k\|^5\right),
\end{align*}
where $\eps>0$ is arbitrary.

Now, for $\mu,\lambda\in \mathbb{R}$,  we  conclude that there exists $C>0$ depending only on $M_0$, $E_0$, $\lambda$, $\mu$, $p$ and $a$ such that
\begin{equation*}
\|\delta_x^+U^k\|\leq C,\, \|U^k\|\leq C, \|U^k\|_{\infty}\leq C,\quad \forall k\ge0.
\end{equation*}

\section{ Local error bound of $R_M^k$ of CNFD:}
\label{appendix_B}

\begin{align*}
r_0^k&= 2\int_0^1\ln(\eps+\sqrt{\sigma|u(0,t_{k+1})|^2+(1-\sigma)|u(0,t_{k})|^2})-\ln(\eps+\sqrt{\frac{|u(0,t_{k+1})|^2+|u(0,t_{k})|^2}{2}})\,d\sigma \\
&+  2\ln(\eps+\sqrt{\frac{|u(0,t_{k+1})|^2+|u(0,t_{k})|^2}{2}})-2\ln(\eps+|u(0,t_{k+\frac{1}{2}})|) \\
&= 2\int_0^1\int_{\frac{1}{2}}^\sigma (\sigma-w)\frac{d^2}{dw}\ln(\eps+f(w))dw\,d\sigma+
2 \ln(\frac{\eps+\sqrt{\frac{|u(0,t_{k+1})|^2+|u(0,t_{k})|^2}{2}}} {\eps+|u(0,t_{k+\frac{1}{2}})|} )\\
&= 2\int_0^1\int_{\frac{1}{2}}^\sigma (\sigma-w)\left( |u(0,t_{k+1})|^2-|u(0,t_{k})|^2 \right)^2\left(\frac{1}{(\eps+f(w))f(w)^3}-\frac{1}{(\eps+f(w))^2f(w)^2}    \right)dw\,d\sigma\\
&+2 \ln(\frac{\eps+\sqrt{\frac{|u(0,t_{k+1})|^2+|u(0,t_{k})|^2}{2}}} {\eps+|u(0,t_{k+\frac{1}{2}})|} )
\end{align*}
where $f(w)=\sqrt{w|u(0,t_{k+1})|^2+(1-w)|u(0,t_{k})|^2},$ thus we can get
\begin{align*}
|r_0^k|&\leq  2\tau^2\max\{\frac{1}{\eps^2},\frac{1}{M_1^4}\}\|\partial_{t}u(0,t_{k+\frac{1}{2}})\|_{\infty}+2\left| \ln(1+\frac{\sqrt{\frac{|u(0,t_{k+1})|^2+|u(0,t_{k})|^2}{2}}-|u(0,t_{k+\frac{1}{2}})| } {\eps+|u(0,t_{k+\frac{1}{2}})|} )\right|\\
&{\text{if $\sqrt{\frac{|u(0,t_{k+1})|^2+|u(0,t_{k})|^2}{2}}\geq |u(0,t_{k+\frac{1}{2}})|$}}\\
&\leq 2\tau^2\max\{\frac{1}{\eps^2},\frac{1}{M_1^4}\}\|\partial_{t}u(0,t_{k+\frac{1}{2}})\|_{\infty}
+2\left| \frac{\frac{|u(0,t_{k+1})|^2+|u(0,t_{k})|^2}{2}-|u(0,t_{k+\frac{1}{2}})|^2 } {(\eps+|u(0,t_{k+\frac{1}{2}})|)(\sqrt{\frac{|u(0,t_{k+1})|^2+|u(0,t_{k})|^2}{2}}+|u(0,t_{k+\frac{1}{2}})| )} \right|\\
&\leq 2\tau^2\max\{\frac{1}{\eps^2},\frac{1}{M_1^4}\}\|\partial_{t}u(0,t_{k+\frac{1}{2}})\|_{\infty}
+2\tau^2\max\{\frac{1}{\eps},\frac{1}{M_1}\}\|\partial_{tt}u(0,t_{k+\frac{1}{2}})\|_{\infty}.
%
\end{align*}
For $\delta_t^+R_M^k$, we focus on the term $\delta_t^+\left( u(0,t_{k+\frac{1}{2}})r_0^k \right)$, and other terms can be estimated by similar procedure of $R_M^k$:
\begin{align*}
&\delta_t^+\left( u(0,t_{k+\frac{1}{2}})r_0^k \right)=\frac{1}{\tau}\left( (u(0,t_{k+\frac{3}{2}})-u(0,t_{k+\frac{1}{2}}))r_0^k+ u(0,t_{k+\frac{3}{2}})(r_0^{k+1}-r_0^k) \right)\\
&\leq \|\partial_{t}u(0,t_{k+\frac{1}{2}})\|_{\infty}|r_0^k|+u(0,t_{k+\frac{3}{2}})\delta_t^+r_0^{k}
\end{align*}
to give the bound for $\delta_t^+R_M^k$, for simplicity of notation, denote
\begin{equation*}
	u^{\theta}(0,t_k)=\theta u(0,t_k+1)+(1-\theta)u(0,t_k),\,u^{\theta}(0,t_{k+\frac{1}{2}})=\theta u(0,t_{k+\frac{3}{2}})+(1-\theta)u(0,t_{k+\frac{1}{2}}),
\end{equation*}
since
\begin{align*}
	&\delta_t^+r_0^{k}=2\delta_t^+\left( q_0^k-\ln(\eps+|u(.,t_{k+\frac{1}{2}})|)  \right)\\
	&=2\left( \delta_t^+q_0^k-\delta_t^+\ln(\eps+|u(.,t_{k+\frac{1}{2}})|)  \right)\\
	&=2\left( \int_0^1\frac{d}{d\theta}\left[ \int_0^1 \ln(\eps+\sqrt{\sigma|u^{\theta}(0,t_{k+1})|^2+(1-\sigma)|u^{\theta}(0,t_{k})|^2}) \,d\sigma\right]\,d\theta -\int_0^1 \frac{d}{d\theta}\ln(\eps+|u^{\theta}(.,t_{k+\frac{1}{2}})|)  \,d\theta \right)\\
	&=2\tau Re\int_0^1 \int_0^1 \frac{\sigma \delta_t^+u(0,t_{k+1})\overline{u^{\theta}(0,t_{k+1})}+ (1-\sigma)\delta_t^+u(0,t_{k})\overline{u^{\theta}(0,t_{k})} }
	{(\eps+\sqrt{\sigma|u^{\theta}(0,t_{k+1})|^2+(1-\sigma)|u^{\theta}(0,t_{k})|^2})\sqrt{\sigma|u^{\theta}(0,t_{k+1})|^2+(1-\sigma)|u^{\theta}(0,t_{k})|^2}} \,d\sigma\\
	&-\frac{\delta_t^+u(0,t_{k+\frac{1}{2}})\overline{u^{\theta}(0,t_{k+\frac{1}{2}})} }
	{(\eps+|u^{\theta}(.,t_{k+\frac{1}{2}})|)|u^{\theta}(.,t_{k+\frac{1}{2}})|}   \,d\theta
\end{align*}
using the mid-point rule, we can get
\begin{align*}
	&\delta_t^+r_0^{k}
	=2\tau Re\int_0^1 \int_0^1 \frac{\sigma \delta_t^+u(0,t_{k+1})\overline{u^{\theta}(0,t_{k+1})}+ (1-\sigma)\delta_t^+u(0,t_{k})\overline{u^{\theta}(0,t_{k})} }
	{(\eps+\sqrt{\sigma|u^{\theta}(0,t_{k+1})|^2+(1-\sigma)|u^{\theta}(0,t_{k})|^2})\sqrt{\sigma|u^{\theta}(0,t_{k+1})|^2+(1-\sigma)|u^{\theta}(0,t_{k})|^2}} \,d\sigma\\
	&-\frac{ \delta_t^+u(0,t_{k+1})\overline{u^{\theta}(0,t_{k+1})}+ \delta_t^+u(0,t_{k})\overline{u^{\theta}(0,t_{k})} }
	{2(\eps+\sqrt{\frac{|u^{\theta}(0,t_{k+1})|^2+|u^{\theta}(0,t_{k})|^2}{2}})\sqrt{\frac{|u^{\theta}(0,t_{k+1})|^2+|u^{\theta}(0,t_{k})|^2}{2}}}
	+\frac{ \delta_t^+u(0,t_{k+1})\overline{u^{\theta}(0,t_{k+1})}+ \delta_t^+u(0,t_{k})\overline{u^{\theta}(0,t_{k})} }
	{2(\eps+\sqrt{\frac{|u^{\theta}(0,t_{k+1})|^2+|u^{\theta}(0,t_{k})|^2}{2}})\sqrt{\frac{|u^{\theta}(0,t_{k+1})|^2+|u^{\theta}(0,t_{k})|^2}{2}}}\\
	&-\frac{ \delta_t^+u(0,t_{k+1})\overline{u^{\theta}(0,t_{k+1})}+ \delta_t^+u(0,t_{k})\overline{u^{\theta}(0,t_{k})} }
	{2(\eps+|u^{\theta}(.,t_{k+\frac{1}{2}})|)|u^{\theta}(.,t_{k+\frac{1}{2}})|}
	+\frac{ \delta_t^+u(0,t_{k+1})\overline{u^{\theta}(0,t_{k+1})}+ \delta_t^+u(0,t_{k})\overline{u^{\theta}(0,t_{k})} }
	{2(\eps+|u^{\theta}(.,t_{k+\frac{1}{2}})|)|u^{\theta}(.,t_{k+\frac{1}{2}})|}\\
	&-\frac{\delta_t^+u(0,t_{k+\frac{1}{2}})\overline{u^{\theta}(0,t_{k+\frac{1}{2}})} }
	{(\eps+|u^{\theta}(.,t_{k+\frac{1}{2}})|)|u^{\theta}(.,t_{k+\frac{1}{2}})|}   \,d\theta\\
	&\leq \tau^2\max\{\frac{1}{\eps^2},\frac{1}{M_1^2}\}\|\partial_{t}u^{\theta}(0,t_{k+\frac{1}{2}})\|_{\infty}
	+
	\tau^2\max\{\frac{1}{\eps},\frac{1}{M_1^3}\}\|\partial_{tt}u^{\theta}(0,t_{k+\frac{1}{2}})\|_{\infty}
	+ \tau^2\max\{\frac{1}{\eps},\frac{1}{M_1^2}\}\|\partial_{tt}(\delta_t^+u(0,t_{k+\frac{1}{2}})\overline{u^{\theta}(0,t_{k+\frac{1}{2}})})\|_{\infty}
\end{align*}

\section{ The bound of $\hat{Q}_j^k$ and $\delta_x^+\hat{Q}_j^k$:} the following estimation is constructed under the condition $\sqrt{\theta |u(x_j,t_{k+1})|^2+(1-\theta)|u(x_j,t_{k})|^2}\geq \sqrt{\theta |U_j^{k+1}|^2+(1-\theta)|U_j^k|^2}$, if not,  we can exchange the places of  the two term to continue the estimation:
\label{appendix_C}
\begin{align*}
\hat{Q}_j^k
=&\int_0^1\ln(\varepsilon+\sqrt{\theta |u(x_j,t_{k+1})|^2+(1-\theta)|u(x_j,t_{k})|^2})\,d\theta \left(e_j^{k+1}+e_j^k\right)\\
&+\int_0^1\ln\left(1+\frac{\sqrt{\theta |u(x_j,t_{k+1})|^2+(1-\theta)|u(x_j,t_{k})|^2}-\sqrt{\theta |U_j^{k+1}|^2+(1-\theta)|U_j^k|^2} }{\varepsilon+\sqrt{\theta |U_j^{k+1}|^2+(1-\theta)|U_j^k|^2}}\right)d\theta(U_j^{k+1}+U_j^k)\\
&\leq \max{(|\ln(\varepsilon)|,|\ln(\varepsilon+M_1)|)} \left(e_j^{k+1}+e_j^k\right)+ \max{(\frac{M_1}{\varepsilon},1)} \left(e_j^{k+1}+e_j^k\right),
\end{align*}
and for $\delta_x^+\hat{Q}_j^k$, similar to the estimate of $\delta_x^+r_{j}^{\eps,m}$ in \eqref{appendix_B}, we can get the results in Lemma \ref{Nonlinearity_1}.

%

\end{appendices}


\end{document}